\documentclass[final]{siamltex}

\usepackage{hyperref} 

\pdfoutput=1

\usepackage[latin1]{inputenc}
\usepackage{mystyle}
\usepackage{url}

\graphicspath{{./images/}}

\hypersetup{  
   bookmarks=true,
   backref=true,
   pagebackref=false,
   colorlinks=true,
   linkcolor=blue,
   citecolor=red,
   urlcolor=blue,
   pdftitle={Entropic Wasserstein Gradient Flows},
   pdfauthor={Gabriel Peyr\'e},
   pdfsubject={Entropic Wasserstein Gradient Flows}
}

\title{Entropic Wasserstein Gradient Flows}

\author{Gabriel Peyr\'e\thanks{CNRS and CEREMADE, Universit\'e Paris-Dauphine, Place du Mar\'echal De Lattre De Tassigny, 75775 PARIS CEDEX 16, FRANCE.} }

\begin{document}

\maketitle

\begin{abstract}
	This article details a novel numerical scheme to approximate gradient flows for optimal transport (i.e. Wasserstein) metrics. 
	These flows have proved useful to tackle theoretically and numerically non-linear diffusion equations that model for instance porous media or crowd evolutions. These gradient flows define a suitable notion of weak solutions for these evolutions and they can be approximated in a stable way using discrete flows. These discrete flows are implicit Euler time stepping according to the Wasserstein metric.
	A bottleneck of these approaches is the high computational load induced by the resolution of each step. Indeed, this corresponds to the resolution of a convex optimization problem involving a Wasserstein distance to the previous iterate.
	Following several recent works on the approximation of Wasserstein distances, we consider a discrete flow induced by an entropic regularization of the transportation coupling. This entropic regularization allows one to trade the initial Wasserstein fidelity term for a Kullback-Leibler divergence, which is easier to deal with numerically. We show how Kullback-Leibler first order proximal schemes, and in particular Dykstra's algorithm, can be used to compute each step of the regularized flow. 
	The resulting algorithm is both fast, parallelizable and versatile, because it only requires multiplications by the Gibbs kernel $e^{-c/\ga}$ where $c$ is the ground cost and $\ga>0$ the regularization strength. On Euclidean domains discretized on an uniform grid, this corresponds to a linear filtering (for instance a Gaussian filtering when $c$ is the squared Euclidean distance) which can be computed in nearly linear time. On more general domains, such as (possibly non-convex) shapes or on manifolds discretized by a triangular mesh, following a recently proposed numerical scheme for optimal transport, this Gibbs kernel multiplication is approximated by a short-time heat diffusion. 
	We show numerical illustrations of this method to approximate crowd motions on complicated domains as well as non-linear diffusions with spatially-varying coefficients. 
\end{abstract}


\section{Introduction}

\subsection{Optimal Transport}

\paragraph{Optimal transport: from theory to applications}

In the last 20 years or so, optimal transport (OT) has emerged as a foundational tool to analyze diverse problems at the interface between variational analysis, partial differential equations and probability. We refer to the book of Villani~\cite{Villani03} for an introduction to these topics. It took more time for this notion to become progressively mainstream in various applications, which is largely due to the high computation cost of the corresponding (static) linear program of Kantorovich~\cite{Kantorovich42} or to the dynamical formulation of Benamou and Brenier~\cite{Benamou2000}. However, one can now found many relevant uses of OT in very diverse fields such as astrophysics~\cite{FrischNaturee}, computer vision~\cite{RubTomGui00}, computer graphics~\cite{Bonneel-displacement}, image processing~\cite{2014-xia-siims}, statistics~\cite{BigotBarycenter} and machine learning~\cite{CuturiSinkhorn}, to name a few.

\paragraph{Entropic regularization}

In order to obtain fast approximations of optimal transport distances (a.k.a. Wasserstein distances), there has been a recent revival of the so-called entropic regularization method. Cuturi~\cite{CuturiSinkhorn} presented this scheme in the machine learning community as a fully parallelizable algorithm which can make the method scalable to large problems. He shows that this corresponds to the application of the well-known iterative diagonal scaling algorithm, which is sometime referred to as Sinkhorn's algorithm~\cite{Sinkhorn64,SinkhornKnopp67,Sinkhorn67} or IPFP~\cite{DemingStephanIPFP}. This method is also closely related to Schrodinger's problem~\cite{Shrodinger31} of projecting a Gibbs distribution on fixed marginal constraints, see~\cite{RuschendorfThomsen,LeonardShrodinger} for recent mathematical accounts on this problem.

The major interest of this entropic approximation is that it allows one to re-cast various OT-related problems as optimizations over the space of probabilities endowed with the Kulback-Leibler divergence. The geometry of this space, as well as the availability of efficient first-order optimization methods, makes this novel formulation numerically more friendly than the original linear program formulation. The price to pay for such simple and efficient approaches is the presence of an extra amount of smoothing (in fact a blurring by the Gibbs kernel) on the obtained results. 

\paragraph{Variational problems involving OT}

These methods have been used to solve various variational optimization problem involving the Wasserstein distance. For instance the computation of Wasserstein barycenters, initially proposed in~\cite{Carlier_wasserstein_barycenter}, has been approximated by entropic regularization in~\cite{CuturiBarycenter}. A more general class of problems, including multi-marginal transport (see~\cite{PassMultimarginal} for recent results on this topic) as well as generalized Euler's flows (see~\cite{BrenierEulerAMS} for a weak formulation of Euler's equations), partial transport (as defined in~\cite{FigalliPartial}) and capacity-constrained transport (as defined in~\cite{JonathanMcCannCapacity}) have been approximated by entropic smoothing in~\cite{BregmanProj2015}. 

Our work goes in the same direction of applying entropic regularization to speed up the computation of OT-related problems. Instead of considering here the minimization of functionals involving the Wasserstein distance, we consider here the minimization of convex functions according to the Wasserstein distance.

\subsection{Previous Works}
\label{sec-previous-works}

\paragraph{Wasserstein flows -- theory}

It is natural to derive various partial differential equations (PDE's) as gradient flows of certain energy functionals. While it is most of time assumed that the flow follows the gradient as defined through the $L^2$ topology on some Hilbert space of functions, it is sometime desirable to consider more complicated metrics. This allows one to capture different PDE's and also sometime to give a precise meaning to weak solutions of these PDE's. One of the most striking example is the computation of gradient flows over spaces of probability distributions (i.e. positive and normalized measures) according to the topology defined by the Wasserstein metric. In this setting, the gradient descent cannot be understood directly as an infinitesimal explicit descent in the direction of some gradient, but rather as a limit of an implicit Euler step, as detailed in Section~\ref{subsect-jko-stepping}. This idea corresponds to the notion of gradient flows in metric spaces exposed in the book~\cite{ambrosio2006gradient}.   

The pioneer paper of Jordan, Kinderlehrer and Otto~\cite{jordan1998variational} shows how one recovers Fokker-Planck  diffusions of distributions when one minimizes entropy functionals according to the $W_2$ Wasserstein metric.
The corresponding method are often referred to as`JKO flows'' in reference to these authors. Since then, many non-linear PDE's have been derived as gradient flows for Wasserstein metrics, including the porous medium equation~\cite{otto2001geometry}, the heat equation on manifolds~\cite{ErbarHeatManifold}, degenerate parabolic PDE's~\cite{agueh2002existence}, Keller-Segel equation~\cite{blanchet2008convergence} and higher order PDE's~\cite{Burger-JKO}. It is also possible to define a suitable notion of minimizing flow that cannot be written as PDE's due to the non-differentiability of the energy functional, a striking example being the model of crowd motion with congestion proposed by~\cite{maury2010macroscopic},

\paragraph{Wasserstein flows -- numerics}

The use of Wasserstein methods to discretize non-linear evolutions is an emerging field of research. The major difficulty lies in the high computational cost induced by the resolution of each step. 

The case of 1-D densities is simpler because the optimal transport metric is a flat metric when re-parameterized using inverse cumulative functions. This idea is used in~\cite{kinderlehrer1999approximation,blanchet2008convergence,blanchet2012optimal,agueh2013one,Matthes1D}. 
In higher dimension, a first class of approaches uses an Eulerian representation of the discretized density (i.e. on a fixed grid). The resulting problem can be solved using interior point methods for convex energies~\cite{Burger-JKO} or some sort of linearization in conjunction with finite elements~\cite{burger2010mixed} or finite volumes~\cite{CarrilloFiniteVolume} schemes. 
A second class of approaches rather uses a Lagrangian representation, which is well adapted to optimal transport where the thought after solution is obtained by warping the density at the previous iterate. This idea is at the heart of several schemes, using discretized warpings~\cite{carrillo2009numerical}, particules systems~\cite{Westdickenberg2010}, moving meshes~\cite{BuddMoving} and a consistent discretization of the gradients of convex functions (i.e. optimal transports)~\cite{JDB-JKO}. 

In this article, we use an Eulerian discretization and intend at approximating flows for energies that are already convex in the usual (Euclidean) sense. The main goal is to provide a fast and quite versatile discretization scheme through the use of an entropic smoothing method.


\paragraph{First order scheme with respect to Bregman divergences}

First order proximal optimization schemes have been recently popularized in image processing and machine learning, due to their simplicity and the low computational cost of each iteration. Each step typically requires the computation of proximal operators, which are defined as strictly convex optimization sub-problems, corresponding to an implicit step according to the $L^2$ distance. We refer to the book~\cite{BauschkeCombettes11} for an overview of this large class of methods and recent developments. Note that these $L^2$ proximal methods have been used to solve the dynamical formulation of OT~\cite{Benamou2000,FPapPeyOud13}. 

Many of these proximal algorithms have been extended when one replaces the $L^2$ metric by more general Bregman divergences. The prototypical algorithm (although rarely applicable in its original form) that has been extended to this divergence setting is the so-called proximal point algorithm~\cite{EcksteinProxPoint} (see~\cite{KiwielProxPoint} for an extension to more general, possibly non-smooth, divergences) which corresponds to iteratively applying the proximal operator of the function to be minimized.

Iterative projections on convex sets is probably the simplest yet useful example of proximal methods. It has been extended to the general setting of Bregman's divergences by Bregman~\cite{bregman1967relaxation}. This scheme actually computes the projection on the intersection of convex sets if these sets are affine, which is a restrictive assumption. The natural extension of iterative projections to generic closed convex sets is the so-called Dykstra's algorithm~\cite{dyk,csis}, which can be interpreted as a block-coordinate optimization on the dual problem. Dykstra's method has been extended to the special case of half-spaces in~\cite{CensorReich-Dykstra} and to generic closed convex sets in~\cite{bauschke-lewis,BregmanCensorReich-Dykstra}. 
Actually, as we show in Section~\ref{subsec-dykstra-bregman}, this result extends to arbitrary proper lower-semicontinuous convex functions (that are not necessarily indicators of closed convex sets). Note that such an extension is well-known for the case of the $L^2$ metric~\cite{BauschkeCombettes-Dykstra}.

While in this paper we only make use of Dykstra's algorithm, it should be noted that many more proximal splitting algorithms are available in this Bregman's divergences setting, such as Douglas-Rachford and ADMM~\cite{WangBanerjee-ADMM}, primal-dual algorithms~\cite{ChambollePock-div} and hybrid proximal point algorithms~\cite{SolodovSvaiterBregman}

\subsection{Contributions}

In this paper, we present a novel numerical scheme to compute approximations of discrete gradient flows for  Wasserstein metrics. The approximation is performed by an entropic smoothing of the original OT distance. Each step is computed as the resolution of a convex but possibly non-smooth optimization problem involving a Kulback-Leibler divergence to some Gibbs kernel. We thus propose in Section~\ref{sec-bregman-prox} to solve it using an extension of Dykstra's algorithm to this class of problems, for which we prove the convergence to the solution. Our main finding is that this scheme is both simple to implement and competitive in term of computational speed, since it only requires multiplications with the Gibbs kernel, which, for many practical scenarios, can be achieved in nearly linear time.  We illustrate in Secton~\ref{sec-numerics} this point by applications to a crowd motion model involving a non-smooth congestion term and to non-linear diffusions with spatially varying coefficients. Lastly, Section~\ref{sec-general-functinons} presents a generalization of the proposed algorithm to the case were several couplings are optimized. We show the usefulness of this generalization to compute the gradient flow of a Wasserstein attraction term with congestion and to compute evolution of several coupled densities. 

The code to reproduce the numerical part of this article is available online\footnote{\url{https://github.com/gpeyre/2015-SIIMS-wasserstein-jko/}.}

\subsection{Notations}

In the following we consider either vectors $p \in \RR^N$ ($N$ being the number of discretization points) that are usually in the probability simplex
\eql{\label{eq-def-simplex}
	\Si_N \eqdef \enscond{p \in \RR_+^N}{\textstyle\sum_{i=1}^N p_i = 1}
}
and couplings, that are matrices $\pi \in \RR_+^{N \times N}$. We denote $\dotp{p}{q}=\sum_{i=1}^N p_i q_i$ the canonical inner product on $\RR^N$ and similarly on $\RR^{N \times N}$.

For some set $\Cc \subset \RR^Q$ (typically $Q=N$ or $Q=N \times N$), we define its indicator function as
\eq{
	\foralls a \in \RR^Q, \quad
	\iota_\Cc(a) \eqdef \choice{
		0 \qifq a \in \Cc, \\
		+\infty \quad \text{otherwise}.
	}
}
To ease notations, we define $\odot$ and $\frac{\cdot}{\cdot}$ as being entry-wise operations, i.e. $a \odot b \eqdef (a_i b_i)_i$ and $\frac{a}{b} \eqdef (a_i/b_i)_i$. We denote as $\ones \eqdef (1,\ldots,1)^T \in \RR^N$ the vector filled with ones. We define
\eql{\label{eq-modulo-2}
	\foralls \ell \in \NN, \quad
	[\ell]_2 \eqdef \choice{
		1 \qifq \ell \text{ is odd, }\\
		2 \qifq \ell \text{ is even. }	
	}
}

We define minus the entropy on both vectors and couplings (and we make this distinction on purpose to ease the description of the proposed methods) as
\begin{align}
	\label{eq-entropy-defn}
	\foralls p \in \RR^N, \quad
	\oE(p) &\eqdef \sum_{i=1}^N p_i (\log(p_i)-1) + \iota_{\RR^+}(p_i),  \\
	\label{eq-entropy-couplings-defn}
	\foralls \pi \in \RR^{N \times N}, \quad
	E(\pi) &\eqdef \sum_{i,j=1}^Q \pi_{i,j} (\log(\pi_{i,j})-1) + \iota_{\RR^+}(\pi_{i,j}),  
\end{align}
with the convention that $0 \log(0)=0$. 

We define the Kulback-Leibler divergence on both vectors and couplings as
\begin{align}
	\label{eq-defn-kl}
	\foralls (p,q) \in \RR_{+}^N \times \RR_{+,*}^N, \quad
	\oKL(p|q) &\eqdef \sum_{i=1}^Q p_i \log\pa{ \frac{p_i}{q_i} } - p_i + q_i, \\
	\foralls (\pi,\xi) \in \RR_+^{N \times N} \times \RR_{+,*}^{N \times N}, \quad
	\KL(\pi|\xi) &\eqdef \sum_{i,j=1}^Q \pi_{i,j} \log\pa{ \frac{\pi_{i,j}}{\xi_{i,j}} } - \pi_{i,j} + \xi_{i,j}.
\end{align}

\section{Entropic Discrete JKO Flows}
\label{sec-entropic-jko}

In this article, we consider discrete flows (i.e. evolutions are discretized in time) of discrete probability distributions (i.e. the space is also discretized, and we deal with finite dimensional problems). More precisely, we consider a computational grid $\{x_i\}_{i=1}^N$ of $N$ points, which can be understood for instance as an uniform grid in a sub-set of $\RR^d$ (in the numerical illustrations of Section~\ref{sec-numerics} we consider $d=2$) or as vertices of a triangulation of a surface. We thus consider discrete probability measures on this set of points, which can be understood as sums of Dirac masses located at the $x_i$'s locations, and which we represent in the following as vectors $p \in \Si_N$ in the simplex as defined in~\eqref{eq-def-simplex}. 

\subsection{Entropic Regularization of Wasserstein Distance}

Discretized optimal transport of such discrete measure is defined according to some ground cost $c \in \RR^{N \times N}$. A typical scenario is when $x_i \in \RR^d$ are points in the Euclidean space and one considers $c_{i,j} = \norm{x_i-x_j}^\al$, corresponding to the definition of Wasserstein distances. The case $\al=1$ corresponds to Monge's original problem, and $\al=2$ to the so-called $W_2$ metric, which is by far the most studied case because of its geometrical properties~\cite{Villani03}. A natural extension is to consider points $x_i$ on a smooth manifold $\Mm$ and to define $c_{i,j}=d_\Mm(x_i,x_j)^\al$ where $d_\Mm$ is the geodesic distance on the manifold. 

Following several recent works (see Sections~\ref{sec-previous-works}), the entropic-regularized transportation distance between $(p,q) \in \Si_N^2$ for a ground cost $c \in \RR^{N \times N}$  is
\eq{
	W_{\ga}(p,q) \eqdef \umin{\pi \in \Pi(p,q)} \dotp{c}{\pi} + \ga E(\pi)
}
for some regularization parameter $\ga \geq 0$, where the set of couplings with prescribed marginals $(p,q)$ is
\eq{
	\Pi(p,q) \eqdef \enscond{\pi \in \RR_+^{N \times N}}{ \pi\ones=p, \pi^T\ones=q }.
}
The case $\ga=0$ corresponds to the classical, un-regularized, optimal transport, and is a linear program. The case $\ga>0$ corresponds to a strictly convex minimization problem, where $E$ plays the role of a barrier function of the positive octant making the optimization problem better conditioned numerically. But there is more than merely a strict-convexification of the original functional, otherwise one could have settle for more classical log-barrier routinely used in interior-point methods~\cite{nesterov1994interior}. 
Algebraic properties of the entropy, and its close relationship with the Kulback-Leibler (KL) divergence~\eqref{eq-defn-kl} (see Section~\ref{subsec-dykstra-bregman} for a precise statement) indeed enables closed-form solutions for the various marginal projections problems encountered in OT problems (see for instance Proposition~\ref{prop-projection}). 

It is important to remind that $W_{\ga}$ is not a distance for $\ga>0$, and we refer to Section~\ref{sec-conclusion} for a discussion on the impact of this deficiency.

\subsection{JKO Stepping}
\label{subsect-jko-stepping}

Following the initial work of~\cite{jordan1998variational} (which gives the name of the method, ``JKO flows''), it is possible to discretize various non-linear PDE's as a gradient flow of a functional $f$ using implicit gradient step with respect to a Wasserstein distance. Our method relies on the idea of replacing the initial Wasserstein metric by its entropic regularized approximation. 

A entropically regularized JKO iteration is an implicit descent descent step with respect to the $W_\ga$ ``metric''. To be consistent with notations introduced in the remaining parts of this article, we thus refer to it as a proximal operator according to $W_\ga$, and its definition reads, for $\tau>0$, 
\eql{\label{eq-def-jko-operator}
	\foralls q \in \RR^N, \quad
	\Prox_{\tau f}^{W_\ga}(q) \eqdef \uargmin{p \in \Si_N} W_\ga(p,q) + \tau f(p).
}
Note that since $p \mapsto W_{\ga}(p,q)$ is a strictly convex and coercive function, this operator is uniquely defined.  

Starting from some fixed discrete density $p_{t=0} \in \Si_N$, one defines the discrete JKO follow as
\eql{\label{eq-smooth-jko}
	\foralls t>0, \quad p_{t+1} \eqdef \Prox_{\tau_t f}^{W_{\ga_t}}(p_t), 
}
where $\tau_t>0$ is the step size, $\ga_t>0$ the entropic regularization parameter. Note that we allow here these parameters to vary during the iterations, although we use fixed parameters in the numerical sections~\ref{sec-numerics} and~\ref{sec-general-functinons}.

When $c_{i,j} = d_\Mm(x_i,x_j)^2$ (the geodesic distance squared on some smooth manifold $\Mm$), $f$ is smooth and $\ga=0$ (no entropic regularization), a formal computation shows that this scheme discretizes, as $(\tau,1/N) \rightarrow 0$, the PDE
\eq{
	\pd{p}{t} = \diverg_\Mm\pa{ p \nabla_\Mm ( f'(p) ) },
}
where $\diverg_\Mm$ and $\nabla_\Mm$ are the gradient and divergence operators on the manifold $\Mm$, and $f'$ is the differential of $f$ (the gradient for the $L^2$ metric on $\Mm$). We refer to~\cite{ErbarHeatManifold} for a proof of this relationship when $f$ is an entropy on a manifold. 

For instance, in the case where $f(p)=\oE(p) + \dotp{p}{w}$, this discrete flow thus discretizes a linear diffusion-advection on the manifold
\eq{
	\pd{p}{t} = \Delta_\Mm(p) + \diverg_\Mm( p z )
	\qwhereq
	z = \nabla_\Mm(w).
}
so that the mass get advected by the vector field $z$.

\subsection{KL Proximal Operators}

In order for the method that we propose to be applicable, the function $f$ must be convex and should be ``simple'' in the sense that one should be able to compute easily its proximal operator for the KL divergence. Similarly to~\eqref{eq-def-jko-operator}, this proximal operator is defined as
\eql{\label{eq-def-kl-prox-operator}
	\foralls q \in \RR^N, \quad
	\Prox_{\tau f}^{\oKL}(q) \eqdef \uargmin{p \in \RR_+^N} \oKL(p|q) + \tau f(p).
}
Note that since $p \mapsto \oKL(p|q)$ is a strictly convex and coercive function, this operator is uniquely defined.
Section~\eqref{sec-numerics} shows two examples of such ``simple'' functions: an indicator of a box constraint (for crowd movement) and generalized entropies (for non-linear diffusions).

The underlying rationale behind the framework we propose in this article is that, while it is in general impossible to directly compute the operator $\Prox_{\tau f}^{W_\ga}$, there are many functionals for which $\Prox_{\tau f}^{\oKL}$ is accessible either in closed form, or through a fast and precise algorithm. We will thus trade the application of a single implicit $W_\ga$ proximal step by the iterative application of several KL implicit proximal steps. 
Note that, in particular, $f$ does not need to be smooth, which is crucial to model non-PDE evolutions such as crowd movements with congestion~\cite{maury2010macroscopic}.   

The main property of the KL proximal operator are recalled in Appendix~\ref{sec-kl-calculus}.

\section{A Bregman Proximal Splitting Approach}
\label{sec-bregman-prox}

In this section, we show how to re-formulate a single entropic regularized JKO step in order to introduce a KL divergence penalty. This is useful to allows for the application of generalized first order proximal methods. 

\subsection{Reformulation as a KL Minimization}

We consider a single time step $t$, and denoting $q \eqdef p_t$ the previous iterate of the flow, one can re-write the JKO stepping operator~\eqref{eq-def-jko-operator} as
\eq{
	\Prox_{\tau f}^{W_{\ga}}(q) = \pi \ones 
} 
where $\pi \in \RR^{N \times N}$ solves the following strictly convex optimization problem
\eql{\label{eq-prox-step}
	\umin{\pi} 
		\dotp{c}{\pi} + \ga E(\pi) + \tau f(\pi \ones) + \iota_{\Cc_q}(\pi)
}
where we introduced the constraint set
\eq{
	\Cc_q \eqdef \enscond{\pi \in \RR^{N \times N}}{\pi^T \ones=q}.
}

The initial formulation~\eqref{eq-prox-step} can be re-cast as
\eql{\label{eq-defn-KL-prbm}
	\umin{\pi} \KL(\pi|\xi) + \phi_1(\pi) + \phi_2(\pi)
	\qwhereq
	\choice{
		\phi_1(\pi) \eqdef \iota_{\Cc_q}(\pi), \\
		\phi_2(\pi) \eqdef \frac{\tau}{\ga} f(\pi \ones),
	}
}
where we defined the Gibbs kernel $\xi$ as
\eq{ 
	\xi \eqdef e^{-c/\ga} \in \RR_{+,*}^{N \times N}.
}

\subsection{Dykstra Algorithm with Bregman Divergences}
\label{subsec-dykstra-bregman}

\paragraph{Bregman divergence and proximal map}

In order to give a more general treatment of optimization problems of the form~\eqref{eq-defn-KL-prbm}, that can be useful beyond the particular context of this article, we consider a generic Bregman divergence $\BregDiv_\Ga$, defined on some convex set $\Dd$.

We follow~\cite{bauschke-lewis} and define a Bregman divergence (see for instance) as
\eq{
	\foralls (\pi,\xi) \in \Dd \times \text{int}(\Dd), \quad
	\BregDiv_\Ga(\pi|\xi) = \Ga(\pi)-\Ga(\xi)-\dotp{\nabla \Ga(\xi)}{\pi-\xi}.
}
where $\Ga$ is a strictly convex function, smooth on $\interop(\Dd)$ where $\Dd=\dom(\Ga)$ such that its Legendre transform 
\eq{
	\Ga^*(\rho) = \umax{\pi \in \Dd} \dotp{\pi}{\rho} - \Ga(\pi), 
} 
is also smooth and strictly convex. In particular, one has that $\nabla \Ga$ and $\nabla \Ga^*$ are bijective maps between $\interop(\Dd)$ and $\interop(\dom(\Ga^*))$ such that $\nabla \Ga^* = \nabla \Ga^{-1}$. 

For $\Ga = \norm{\cdot}^2$, one recovers the Euclidean norm $\BregDiv_\Ga = \norm{\cdot}^2$. 
One has $\KL = \BregDiv_\Ga$ for $\Ga(\pi)=E(\pi)$, which is cased we used to tackle~\eqref{eq-defn-KL-prbm}. Note that in general, $\BregDiv_\Ga$ is not symmetric and does not satisfy the triangular inequality, so that it is not a distance. We refer to~\cite{bauschke-lewis} for a table detailing many examples of Bregman's divergences. 

Let us write the general form of problem~\eqref{eq-defn-KL-prbm} as
\eql{\label{eq-generic-optim}
	\umin{\pi \in \Dd} \BregDiv_\Ga(\pi|\xi) + \phi_1(\pi) + \phi_2(\pi)
}
where $\phi_1,\phi_2$ are two proper, lower-semicontinuous convex functions defined on $\Dd$. We also assume that the following qualification constraint holds
\eql{\label{eq-qualif}
	\ri( \dom(\phi_1) ) \cap \ri(\dom(\phi_2)) \cap \ri(\dom(\Gamma) )  \neq \emptyset.
}
where $\ri$ is the relative interior and $\dom(\phi)=\enscond{\pi}{\phi(\pi) \neq +\infty}$.

We define the proximal map of a convex function $\phi$ according to this divergence as
\eql{\label{eq-defn-proxKL}
	\Prox^{\BregDiv_\Ga}_{\phi}(\pi) \eqdef \uargmin{\tilde \pi \in \Dd} \BregDiv_\Ga(\tilde \pi|\pi) + \phi(\tilde \pi).
}
We assume that $\phi$ is coercive, so that $\Prox^{\BregDiv_\Ga}_{\phi}(\pi)$ is always uniquely defined by strict convexity. Furthermore, one has $\Prox^{\BregDiv_\Ga}_{\phi}(\pi) \in \interop(\Dd)$, see~\cite{bauschke-lewis}.

\paragraph{Dykstra's iterations}

Dykstra's algorithm starts by initializing  
\eq{
	\itz{\pi} \eqdef \xi
	\qandq
	\itz{v}=v^{(-1)} \eqdef 0.
}
One then iteratively defines, for $\ell>0$
\begin{align}\label{eq-it-bregman-1}
	\iter{\pi} &\eqdef  \Prox^{\BregDiv_\Ga}_{\phi_{[\ell]_2}}(  \nabla \Ga^*( \nabla \Ga(\itA{\pi}) + \itAA{v} ) ), \\
	\label{eq-it-bregman-2}
	\iter{v} &\eqdef \itAA{v} + \nabla\Ga( \itA{\pi} ) - \nabla \Ga(\iter{\pi}),
\end{align}
where $[\ell]_2$ is defined in~\eqref{eq-modulo-2}. Note that the iterates satisfies $\iter{\pi} \in \interop(\Dd)$, so that the algorithm is well defined. 

The iterates $\iter{\pi}$ of this algorithm are known to converge to the solution of~\eqref{eq-generic-optim} in the case where $\phi_1$ and $\phi_2$ are indicators of convex sets, see~\cite{bauschke-lewis}. This corresponds to the case where $\Prox^{\BregDiv_\Ga}_{\phi_{i}}$ for $i=1,2$ are projectors according to the Bregman divergence.

\paragraph{Convergence proof}

This convergence result in fact caries over to the more general setting where $(\phi_1,\phi_2)$ are arbitrary proper and lower-semicontinuous convex functions.  The proof follows from the fact that Dykstra's iterations correspond to an alternate block minimization algorithm on the dual problem. This idea was suggested to us by Antonin Chambolle and Jalal Fadili. 

\begin{prop}\label{prop-conv-dykstra}
	If condition~\eqref{eq-qualif} holds, then $\iter{\pi}$ converges to the solution of~\eqref{eq-generic-optim}.
\end{prop}
\begin{proof}
The dual problem to~\eqref{eq-generic-optim} reads
\eql{\label{eq-generic-optim-dual}
	\umax{u_1,u_2} -\phi_1^*(u_1)-\phi_2^*(u_2)-\Ga^*( \al-u_1-u_2 ) - C(\xi)
}
where the constant is $C(\xi) \eqdef \dotp{\al}{\xi}-\Ga(\xi)$
and where we defined $\al \eqdef \nabla\Ga(\xi)$.

Duality means that under the domain qualification hypothesis~\eqref{eq-qualif}, the minimum value of~\eqref{eq-generic-optim} and the maximum value of~\eqref{eq-generic-optim-dual} are the same, and that the primal solution $\pi$ can be recovered from the dual one $(u_1,u_2)$ as
\eql{\label{eq-primal-dual-bregman}
	\pi = \nabla\Ga^*(-u_1-u_2). 
}

Starting from $(\itz{u_1},\itz{u_2})=(0,0)$, the alternate block optimization  on~\eqref{eq-generic-optim-dual} defines a sequence $(\iter{u_1},\iter{u_2})$, where, denoting $i=[\ell]_2$ (as defined in~\eqref{eq-modulo-2}) and $j=3-i \in \{1,2\}$, the update at iteration $\ell$ reads
\eql{\label{eq-block-min-dykstra}
	\iter{u_j} \eqdef \itA{u_j}
	\qandq
	\iter{u_i} \eqdef \uargmax{u_i} -\phi_i^*(u_i) - \Ga^*( q-u_i )
}
where we defined $q \eqdef \al-\itA{u_j}$.

Since in~\eqref{eq-generic-optim-dual} the coupling term $\Ga^*( \al-u_1-u_2 )$ between $(u_1,u_2)$ is smooth, a classical result ensures that $(\iter{u_1},\iter{u_2})$ converges to the solution $(u_1^\star,u_2^\star)$ of~\eqref{eq-generic-optim-dual}, see for instance~\cite{Ciarlet-Book}.
 
The primal problem associated to the dual maximization~\eqref{eq-block-min-dykstra} is
\eql{\label{eq-block-min-dykstra-primal}
	\umin{\pi_i} \Ga(\pi_i) - \dotp{q}{\pi_i} + \phi_i(\pi_i) = \BregDiv_\Ga( \pi_i|\nabla\Ga^*(q) ) + \phi_i(\pi_i) + C
}
where $C \in \RR$ is a constant.
The primal-dual relationship between the solutions of~\eqref{eq-block-min-dykstra} and~\eqref{eq-block-min-dykstra-primal} reads 
\eql{\label{eq-rel-primal-dual-proofdyk}
	\pi_i = \nabla\Ga^*(q-u_i).
}
Equations~\eqref{eq-block-min-dykstra} and~\eqref{eq-rel-primal-dual-proofdyk} show that one has
\eql{\label{eq-dyk-dual-iter}
	\iter{u_i} = \al-\itA{u_j} - \nabla \Ga \circ \Prox_{\phi_i}^{\BregDiv_\Ga}\pa{
		\nabla\Ga^*( \al - \itA{u_j} )
	}.
}

We now perform the following change of variables $(\iter{u_1},\iter{u_2}) \rightarrow (\iter{\pi},\iter{v})$
\eq{
	\iter{\pi} = \choice{
		\nabla\Ga^*( \al - \iter{u_1}-\itA{u_2} ) \qifq [\ell]_2=1, \\
		\nabla\Ga^*( \al - \itA{u_1}-\iter{u_2} ) \qifq [\ell]_2=2, 
	}
	\qandq
	\iter{v} = -\iter{u_{[\ell]_2}}.
}
One then verifies that these variables satisfy the relationship~\eqref{eq-it-bregman-2} and that~\eqref{eq-dyk-dual-iter} is equivalent to~\eqref{eq-it-bregman-1}. This shows by recursion that $(\iter{\pi},\iter{v})$ corresponds to Dykstra's variables. The convergence of $(\iter{u_1},\iter{u_2})$ toward $(u_1^\star,u_2^\star)$ implies that $\iter{\pi}$ converges to $\pi^\star \eqdef \nabla\Ga^*( \al - u_1^\star-u_2^\star )$ which is the solution of~\eqref{eq-generic-optim} thanks to the primal-dual relationship~\eqref{eq-primal-dual-bregman}.
\end{proof}

\subsection{Dykstra's Algorithm for $\KL$ divergence}

We now consider the case where $\Ga=E$, $\BregDiv_\Ga=\KL$. To ease the notations, we make the change of variables $\iter{z} \eqdef \nabla \Ga(\iter{v})$. One has that $\nabla \Ga = \log$ and $\nabla \Ga^* = \exp$ and thus one has the iterates 
\begin{align}
	\itz{\pi} &\eqdef \xi \qandq	\itz{z}=z^{(-1)} \eqdef \ones \\
	\label{eq-iter-dystra-1}
	 \foralls \ell>0, \quad \iter{\pi} &\eqdef  \Prox^{\KL}_{\phi_{[\ell]_2}}( \itA{\pi} \odot \itAA{z} ), \\
	\label{eq-iter-dystra-2}
	\iter{z} &\eqdef \itAA{z} \odot \frac{ \itA{\pi} }{ \iter{\pi} }.
\end{align}
Recall here that $\odot$ and $\frac{\cdot}{\cdot}$ denotes entry-wise operations.

\subsection{KL Proximal Operator for JKO Stepping}
\label{sec-kl-jko-algo}

In order to be able to apply iterations~\eqref{eq-iter-dystra-1} and~\eqref{eq-iter-dystra-2}, one needs to be able to compute the proximal operator for the $\KL$ divergence of $\phi_1$ and $\phi_2$. 

The following proposition shows that these proximal operators for the $\KL$ divergence can be indeed computed in closed form as long as one can compute in closed for the proximal operator of $f$ for the $\oKL$ divergence.

\begin{proposition}\label{prop-projection}
	For any $\pi \in \RR_+^{N \times N}$, one has 	
	\eql{\label{eq-formula-prox-explicit}
		\Prox_{\phi_1}^{\KL}(\pi) = \pi \diag\pa{ \frac{q}{\pi^T \ones} }
		\qandq
		\Prox_{\phi_2}^{\KL}(\pi) = \diag\pa{
					\frac{ \Prox_{\frac{\tau}{\ga} f}^{\oKL}(\pi \ones) }{\pi\ones}
				} \pi
	}
\end{proposition}
\begin{proof}
	The computation of $\Prox_{\phi_1}^{\KL}$ is obtained by combining~\eqref{eq-prox-calculus-lifting-transp} and~\eqref{eq-prox-calculus-indic} in the special case $M=1$.
	The computation of $\Prox_{\phi_2}^{\KL}$ is obtained by applying~\eqref{eq-prox-calculus-lifting} in the special case of $M=1$ coupling.
\end{proof}

\subsection{Dykstra Algorithm for JKO Stepping}
\label{subsec-dykstra-jko}

Writing down the first order optimality conditions with respect to $\pi$ for problem~\eqref{eq-prox-step} shows that there exists $(a,b) \in (\RR_+^N)^2$ such that the optimal $\pi$ satisfies $\pi = \diag(a)\xi\diag(b)$. It means that, just as for the classical entropic regularization of optimal transport~\cite{CuturiSinkhorn}, the optimal coupling $\pi$ is a diagonal scaling of the initial Gibbs kernel $\xi$. This remark actually not only holds for the optimal $\pi$, but it also holds for each iterate $\iter{\pi}$ constructed by iterations~\eqref{eq-iter-dystra-1} and~\eqref{eq-iter-dystra-2} that defines $(\iter{\pi}, \iter{z}) \in (\RR_+^N)^2$.

The following proposition makes use of this remark and shows how to actually implement numerically iterations~\eqref{eq-iter-dystra-1} and \eqref{eq-iter-dystra-2} of the method in a fast and parallel way using only matrix-vector multiplications against the kernel $\xi$.

\newcommand{\myprox}{\ProxKL}
\renewcommand{\myprox}{\text{Prox}}

\begin{proposition}\label{prop-implementation-scaling}
The iterates of Dykstra's algorithm can be written as 
\eql{\label{eq-factor-format}
	\iter{\pi} = \diag(\iter{a}) \xi \diag(\iter{b})
	\qandq
	\iter{z} = \iter{u} v^{(\ell),T}
}
(i.e. $\iter{z}$ is a rank-1 matrix) where	$(\iter{a},\iter{b},\iter{u},\iter{v}) \in (\RR_{+,*}^N)^4$,  
with the initialization 
\eql{\label{eq-init-factor}
	\itz{a}=\itz{b}=\itz{u}=\itz{v}=\ones.
}
For odd $\ell$, the update of $(\iter{a},\iter{b})$ reads
\eql{\label{eq-update-formula-1}
	\iter{a} = \itA{a} \odot \itAA{u}
	\qandq
	\iter{b} = \frac{q}{\xi^T(\iter{a})}, 	
}
while for even $\ell$ it reads
\eql{\label{eq-update-formula-2}
	\iter{b} = \itA{b} \odot \itAA{v}
	\qandq 
	\iter{a} = \frac{
			\iter{p}
		}{\xi(\iter{b})}, 
}
where we defined 
\eql{\label{eq-defn-pell}
	\iter{p} \eqdef \Prox_{\frac{\tau}{\ga} f}^{\KL}( \itA{a} \odot \itAA{u} \odot \xi(\iter{b})).
}
The update of $(\iter{u},\iter{v})$ is always
\eql{\label{eq-update-formula-3}
	\iter{u}  = 	\itAA{u} \odot \frac{ \itA{a} }{ \iter{a} }
	\qandq
	\iter{v} = \itAA{v} \odot \frac{ \itA{b} }{ \iter{b} }.
}
\end{proposition}
\begin{proof}
	One verifies that the format~\eqref{eq-factor-format} holds for the initialization~\eqref{eq-init-factor}
	and that it is maintained by the update formulas~\eqref{eq-formula-prox-explicit}. Formulas~\eqref{eq-update-formula-1}, \eqref{eq-update-formula-2} and \eqref{eq-update-formula-3} are obtained by identifying the different terms when plugging the format~\eqref{eq-factor-format} into the update formulas~\eqref{eq-formula-prox-explicit}.
\end{proof}

The Pseudo-code~\ref{table-pseudo-code} recaps all the successive steps needed to compute the full JKO flow~\eqref{eq-smooth-jko} with entropic smoothing. 
This resolution thus only requires to iteratively apply, until a suitable convergence criterion is met, the update rules~\eqref{eq-update-formula-1}, \eqref{eq-update-formula-2} and~\eqref{eq-update-formula-3}. In practice, we found that monitoring the violation of the constraint $\Cc_q$ to be both a simple and efficient way to enforce a stopping criterion This criterion allows furthermore to precisely enforce mass conservation, i.e. $p_t \in \Si_N$ stays normalized to unit mass, which is important in many practical cases. 

The crux of the method, that is extensively used in the numerical section (see in particular Section~\ref{subsec-kernel-comp}) is that one only needs to know how to apply the kernel $\xi$ and its adjoint $\xi^T$ (which are in most practical situations equals), which can be achieved either exactly or approximately in fast and highly parallelizable manner.

\begin{listing}[h!]
	\begin{enumerate}
		\item Initialize $t=0$ and $p_{t=0}$. 
		\item Initialize $\ell=1$ and set
		\eq{
			\itz{a}=\itz{b}=\itz{u}=\itz{v}=\ones.
		} 
		\item Setting $q \eqdef p_t$, update $(a^{(\ell)},b^{(\ell)})$ using~\eqref{eq-update-formula-1} is $\ell$
			is odd, and using~\eqref{eq-update-formula-2} if $\ell$ is even.		
		\item Update $(u^{(\ell)},v^{(\ell)})$ using~\eqref{eq-update-formula-3}.
		\item If $\norm{\iter{b} \odot \xi^T(\iter{a}) - q} > \epsilon$ or if $\ell$ is odd, 
			set $\ell \leftarrow \ell+1$
			and go back to step 3.
		\item Set $p_{t+1} = \iter{p}$ as defined by~\eqref{eq-defn-pell}, $t \leftarrow t+1$ and go back to step 2.
	\end{enumerate}
	\caption{%
		Iterations computing the full JKO flow. 
		The inputs are the initial density $p_{t=0}$, the parameters $(\ga,\tau)$ and the tolerance $\epsilon$.
		The outputs are the iterates $(p_t)_{t > 0}$.
	}
   \label{table-pseudo-code}
\end{listing}

\section{Numerical Results}
\label{sec-numerics}

We now illustrate the usefulness and versatility of our approach to compute approximate solutions to various non-linear diffusion processes. The videos showing the time evolutions displayed in the figures bellow are available online\footnote{\url{https://github.com/gpeyre/2015-SIIMS-wasserstein-jko/tree/master/videos}}.

\subsection{Exact and Approximate Kernel Computation}
\label{subsec-kernel-comp}

As already highlighted in Section~\ref{subsec-dykstra-jko}, our method is efficient if one can compute in a fast way the multiplication $\xi p$ between the Gibbs kernel $\xi=e^{-c/\ga}$ and a vector $p \in \RR^N$.  In the general case, this is intractable because this is a full matrix-vector multiplication. Even if $\xi$ usually has an exponential decay away from the diagonal, precisely capturing this decay is important to achieved transportation of mass effects. In particular, truncating the kernel to obtain a sparse matrix is prohibited. 

\paragraph{Translation invariant metrics}

The simplest setting is when the sampling points $(x_i)_{i=1}^N$ (as defined in Section~\ref{sec-entropic-jko}) correspond to an uniform grid discretization of a translation invariant distance, i.e. $c_{i,j} = D(x_i-x_j)^\al$ for  some function $D : \RR^d \rightarrow \RR$. In this case, $\xi$ is simply a discrete convolution against the kernel $D(\cdot)^\al$ sampled on an uniform grid. For instance, when $D(\cdot) = \norm{\cdot}$ and $\al=2$, $\xi$ is simply a convolution with a Gaussian kernel of width $\ga$. When using periodic or Neumann boundary conditions, it is thus possible to implement this convolution in $O(N\log(N))$ operations using Fast Fourier Transforms (FFT's). There also exists a flurry of linear time approximate convolutions, the most popular one being Deriche's factorization with IIR filters~\cite{deriche-1993}. We used this method to generate Figures~\ref{fig-influ-kappa} and~\ref{fig-porous}. The other figures require a more advanced treatment because the kernel is not translation invariant. We now detail this approach.

\paragraph{Riemannian metrics}

Unfortunately, many case of practical interest correspond to diffusions inside non-convex domains, or even on non-Euclidean domains, typically a Riemannian manifold $\Mm$ (possibly with a boundary). In this setting, the natural choice for the ground cost $c$ is to set $c_{i,j} = d_\Mm(x_i,x_j)^\al$, where $d_\Mm$ is the geodesic distance on the manifold. 
A major issue is that computing this matrix $c$ is costly, for instance it would take $O(N^2\log(N))$ using Fast-Marching technics~\cite{sethian-book} on a grid or a triangulated mesh of $N$ points. Even storing this non-sparse matrix can be prohibitive, and applying it at each step of the Dykstra algorithm would require $O(N^2)$ operations. 
Inspired by the ``geodesic in heat'' method of~\cite{crane-2013}, it has recently been proposed by~\cite{ConvolutionalOT} to speed up these computations by approximating the kernel $\xi$ by a Laplace or a heat kernel $\tilde\xi$ (depending on wether $\al=1$ or $\al=2$). This means that $c$ does not need to be pre-computed and stored, and that, as explained bellow, one can apply it on the fly at each iteration using a fast sparse linear solver. 
In the numerical tests, we have used this approximation. 

To this end, one only needs to have at its disposal a discrete approximation $\De_\Mm$ of the Laplacian operator on the manifold $\Mm$. This is easily achieved using axis-aligned finite differences for manifold discretized on a rectangular grid, and this is the case for Figures~\ref{fig-examples}. When considering a discretized manifold $\Mm$ which is a triangulated surface (as this is the case for Figure~\ref{fig-meshes}), one can use piecewise linear $P_1$ finite elements, and the operator $\De_\Mm$ is then the popular Laplacian with cotangent weights, see~\cite{botsch-2010}. 

Following~\cite{ConvolutionalOT}, the kernel $\xi$ is then approximated using $L \in \NN^*$ steps of implicit Euler time discretization of the resolution of the heat equation on $\Mm$ until time $\ga$, i.e.
\eql{\label{eq-heat-kernel}
	\xi \approx \tilde\xi \eqdef \pa{ \Id - \frac{\ga}{L} \Delta_\Mm }^{-L} 
}
where $(\cdot)^{-L}$ means that one iterates $L$ times matrix inversion. 

When $L=1$, and ignoring discretization errors, the result of Varadhan~\cite{varadhan-1967} shows that in the limit of small $\ga$, $\tilde\xi$ converges to the kernel $\xi$ obtained when using $\al=1$ (i.e. ``$W_1$'' optimal transport). As $L$ increases, $\tilde\xi$ converges to the heat kernel, which can be shown, also using~\cite{varadhan-1967} to be consistent with the case $\al=2$ (i.e. ``$W_2$'' optimal transport). In the numerical tests, we have used $L=10$. 

Numerically, the computation of matrix/vector multiplications $\tilde\xi p$ appearing the Dykstra updates thus requires the resolution of $L$ linear systems. Since these matrix/vector multiplications are computed many times during the iterations, a considerable speed-up is achieved by first pre-computing a sparse Cholesky factorization of $\Id - \frac{\ga}{L} \Delta_\Mm$ and then solving the $L$ linear systems by back-substitution~\cite{Davis:2006}. Although there is no theoretical guarantees, we observed numerically that each step typically has a linear time complexity because the factorization is indeed highly sparse. We refer to~\cite{crane-2013} for an experimental analysis of this class of Laplacian solvers.  

\subsection{Crowd Motion Model}

To model crowd motion, we follow~\cite{maury2010macroscopic}, where a congestion effect (not too many peoples can be at the same position) is enforced by imposing that the density $p$ of peoples follows a JKO flow with the functional $f$ defined as
\eql{\label{eq-dfn-congestion}
	\foralls p \in \RR^N, \quad
	f(p) \eqdef \iota_{[0,\kappa]^N}(p) + \dotp{w}{p}
}
where $\kappa \geq \normi{p_{t=0}}$ is the congestion parameter (the smaller, the more congestion)
and $w \in \RR^N$ is a potential (the mass should move along the gradient of $w$).

\newcommand{\myfigM}[3]{\includegraphics[width=.19\linewidth]{#1-kappa#3/#1-kappa#3-#2}}

\begin{figure}[h!]
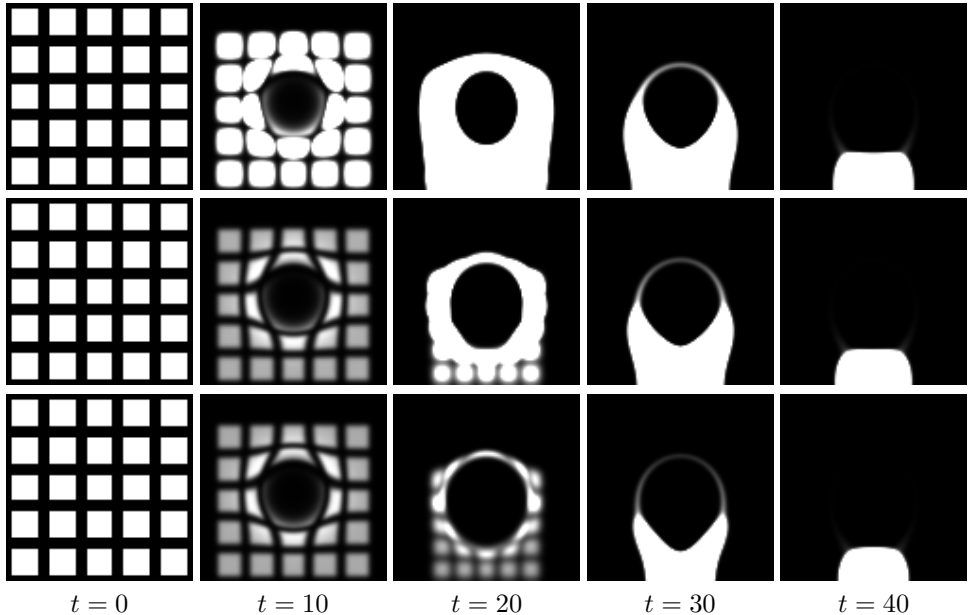

	\centering
	\begin{tabular}{@{}c@{\hspace{1mm}}c@{\hspace{1mm}}c@{\hspace{1mm}}c@{\hspace{1mm}}c@{}}
		\myfigM{bump}{1}{10}&
		\myfigM{bump}{2}{10}&
		\myfigM{bump}{5}{10}&
		\myfigM{bump}{10}{10}&
		\myfigM{bump}{20}{10}\\
		\myfigM{bump}{1}{20}&
		\myfigM{bump}{2}{20}&
		\myfigM{bump}{5}{20}&
		\myfigM{bump}{10}{20}&
		\myfigM{bump}{20}{20}\\
		\myfigM{bump}{1}{40}&
		\myfigM{bump}{2}{40}&
		\myfigM{bump}{5}{40}&
		\myfigM{bump}{10}{40}&
		\myfigM{bump}{20}{40}\\
		$t=0$ & $t=10$ & $t=20$ & $t=30$ & $t=40$ 
	\end{tabular}
	\caption{%
		Display of the influence of the congestion parameter $\kappa$ on the evolution (the driving potential $w$ is displayed on the upper-right corner of Figure~\ref{fig-examples}). 
		From top to bottom, the parameters are set to $\kappa/\normi{p_{t=0}} \in \{1, 2, 4\} \}$.
	}
   \label{fig-influ-kappa}
\end{figure}

For such a function, the KL proximal operator is easy to compute, as detailed in the following proposition.

\begin{prop}\label{prop-prox-congest}
	One has 
	\eql{\label{eq-dfn-congestion-prox}
		\foralls p \in \RR^N, \quad 
		\Prox_{\si f}^{\oKL}(p) = \min(p \odot e^{-\si w},\kappa)
	}	
	where the min should be understood components-wise. 
\end{prop}
\begin{proof}
	The formula when $w=0$ is easy to show, and one then apply~\eqref{eq-prox-calculus-shift} in the case $M=1$.
\end{proof}

Note that it is of course possible to consider a $\kappa$ that is spatially varying to model a non-homogeneous congestion effect. 

Figure~\ref{fig-influ-kappa} shows the influence of the congestion parameter $\kappa$. This figure is obtained for the quadratic Euclidean cost $c_{i,j}=\norm{x_i-x_j}^2$ on a $N = 200 \times 200$ uniform grid with Neumann boundary conditions. The kernel $\xi$ is computed using a fast Gaussian filtering on this grid as detailed in Section~\ref{subsec-kernel-comp}.  

Figure~\ref{fig-examples} shows various evolutions for different potentials $w$ (guiding the crowd to the exit) on a manifold $\Mm$ which is a sub-set of a square in $\RR^2$. This means that locally the Riemannian metric is Euclidean, but since the domain is non-convex, the transport is defined according to a geodesic distance $d_\Mm$ which is not the Euclidean distance. The discretization is achieved using the approximate heat kernel~\eqref{eq-heat-kernel} with $L=10$ and on a grid of $N=100 \times 100$ points. 

\newcommand{\myfig}[2]{\includegraphics[width=.16\linewidth]{#1-kappa10/#1-kappa10-#2}}
\newcommand{\myfigPot}[1]{\includegraphics[width=.16\linewidth]{potentials/#1-potential}}

\begin{figure}[h!]
	\centering
	\begin{tabular}{@{}c@{\hspace{1mm}}c@{\hspace{1mm}}c@{\hspace{1mm}}c@{\hspace{1mm}}c@{\hspace{1mm}}c@{}}
		\myfig{bump}{1}&
		\myfig{bump}{2}&
		\myfig{bump}{5}&
		\myfig{bump}{10}&
		\myfig{bump}{20}&
		\myfigPot{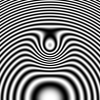} \\
		\myfig{disk}{1}&
		\myfig{disk}{2}&
		\myfig{disk}{5}&
		\myfig{disk}{10}&
		\myfig{disk}{20}&
		\myfigPot{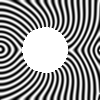} \\
		\myfig{disk1}{1}&
		\myfig{disk1}{2}&
		\myfig{disk1}{5}&
		\myfig{disk1}{10}&
		\myfig{disk1}{20}&
		\myfigPot{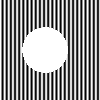} \\
		\myfig{holes}{1}&
		\myfig{holes}{2}&
		\myfig{holes}{5}&
		\myfig{holes}{10}&
		\myfig{holes}{20}&
		\myfigPot{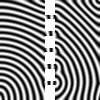} \\
		\myfig{tworooms}{1}&
		\myfig{tworooms}{2}&
		\myfig{tworooms}{5}&
		\myfig{tworooms}{10}&
		\myfig{tworooms}{20}&
		\myfigPot{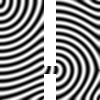}\\
		$t=0$ & $t=10$ & $t=20$ & $t=30$ & $t=40$ & $\cos(w)$
	\end{tabular}
	\caption{%
		Display of crowd evolution for $\kappa=\normi{p_{t=0}}$. 
		The rightmost column display equispaced level-sets of the driving potential $w$. 
	}
   \label{fig-examples}
\end{figure}

Lastly Figure~\ref{fig-meshes} shows the evolution on a triangulated mesh of $20000$ vertices, which is also implemented using the same heat kernel, but this time on a 3-D triangulation using piecewise linear finite elements (hence a discrete Laplacian with cotangent weights~\cite{botsch-2010}).

\renewcommand{\myfigPot}[1]{\includegraphics[width=.19\linewidth]{potentials/#1-potential}}

\begin{figure}[h!]
	\centering
	\begin{tabular}{@{}c@{\hspace{1mm}}c@{\hspace{1mm}}c@{\hspace{1mm}}c@{\hspace{1mm}}c@{}}
		\multirow{2}{*}[3em]{\myfigPot{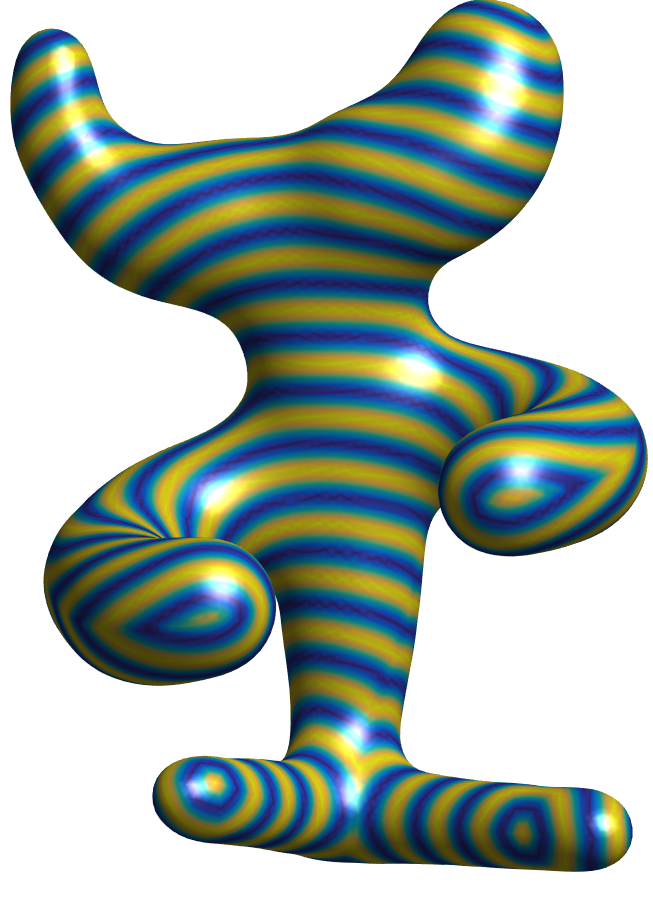}} &
		\myfigM{moomoo}{1}{10}&
		\myfigM{moomoo}{5}{10}&
		\myfigM{moomoo}{10}{10}&
		\myfigM{moomoo}{20}{10}\\
		&
		\myfigM{moomoo}{1}{60}&
		\myfigM{moomoo}{5}{60}&
		\myfigM{moomoo}{10}{60}&
		\myfigM{moomoo}{20}{60} \\
		$\cos(w)$ & $t=0$ & $t=15$ & $t=30$ & $t=45$ 
	\end{tabular}
	\caption{%
		Display of the evolution $p_t$ on a triangulated surface. 
		From top to bottom, the congestion parameter is set to $\kappa/\normi{p_{t=0}} \in \{1, 6\} \}$. 
	}
   \label{fig-meshes}
\end{figure}

\subsection{Anisotropic Diffusion Kernels}
\label{sec-anisotropic}

We consider the crowd motion functional~\eqref{eq-dfn-congestion} over measures defined on $\Mm = \RR^2$ now equipped with a Riemannian manifold structure defined by some tensor field $T(x) \in \RR^{d \times d}$ of symmetric positive matrices. We use the heat kernel approximation detailed in Section~\ref{subsec-kernel-comp}. The kernel~\eqref{eq-heat-kernel} thus corresponds to a discretization of an anisotropic diffusion, which are routinely used to perform image restoration~\cite{WeickertBook}. As the anisotropy (i.e. the maximum ratio between the maximum and minium eigenvalues) of the tensors increases, the corresponding linear PDE becomes ill-posed, and traditional discretizations using finite differences are inconsistent, leading to unacceptable artifacts. We thus use the adaptive anisotropic stencils recently proposed in~\cite{FehrenbachMirebeau} to define the sparse Laplacian matrix discretizing the manifold Laplacian $\Delta_\Mm u(x) = \text{div}( T(x) \nabla u(x) )$. This discrete Laplacian is able to cope with highly anisotropic tensor fields. This is illustrated in Figure~\ref{fig-anisotropic}, which shows the impact of the anisotropy on the trajectory of the mass. The potential $w$ creates an horizontal movement of the mass, but the circular shape of the tensor orientations forces the mass to rather follow a curved trajectory. Ultimately, mass accumulates on the left side and the congestion effect appears.

\renewcommand{\myfig}[2]{\includegraphics[width=.16\linewidth]{#1-kappa10/#1-kappa10-#2}}
\renewcommand{\myfigPot}[1]{\includegraphics[width=.16\linewidth]{potentials/#1-potential}}

\begin{figure}[h!]
	\centering
	\begin{tabular}{@{}c@{\hspace{1mm}}c@{\hspace{1mm}}c@{\hspace{1mm}}c@{\hspace{1mm}}c@{\hspace{1mm}}c@{}}
		\myfig{aniso1}{1}&
		\myfig{aniso1}{5}&
		\myfig{aniso1}{10}&
		\myfig{aniso1}{15}&
		\myfig{aniso1}{20} &
		\myfigPot{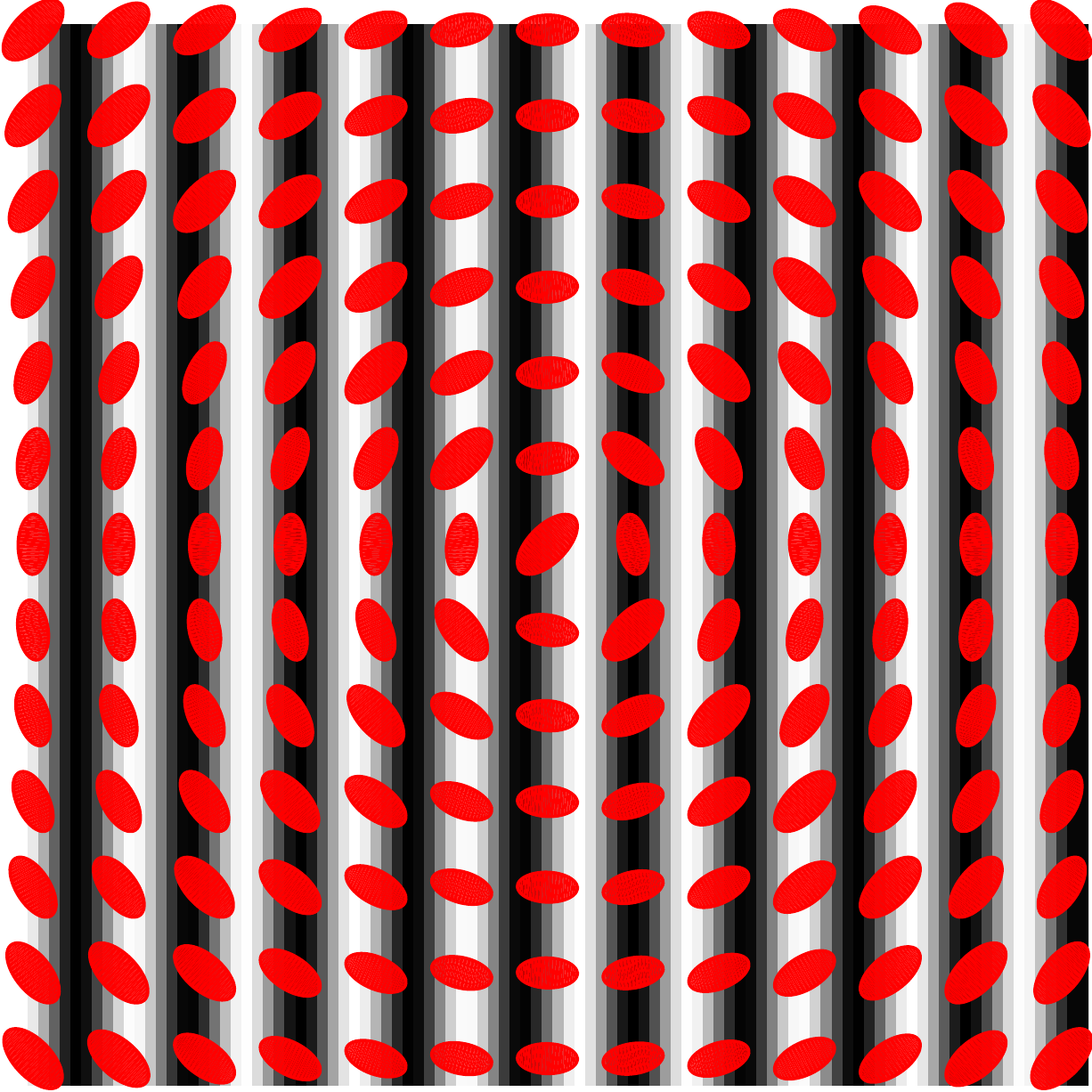} \\
		\myfig{aniso3}{1}&
		\myfig{aniso3}{5}&
		\myfig{aniso3}{10}&
		\myfig{aniso3}{15}&
		\myfig{aniso3}{20} & 
		\myfigPot{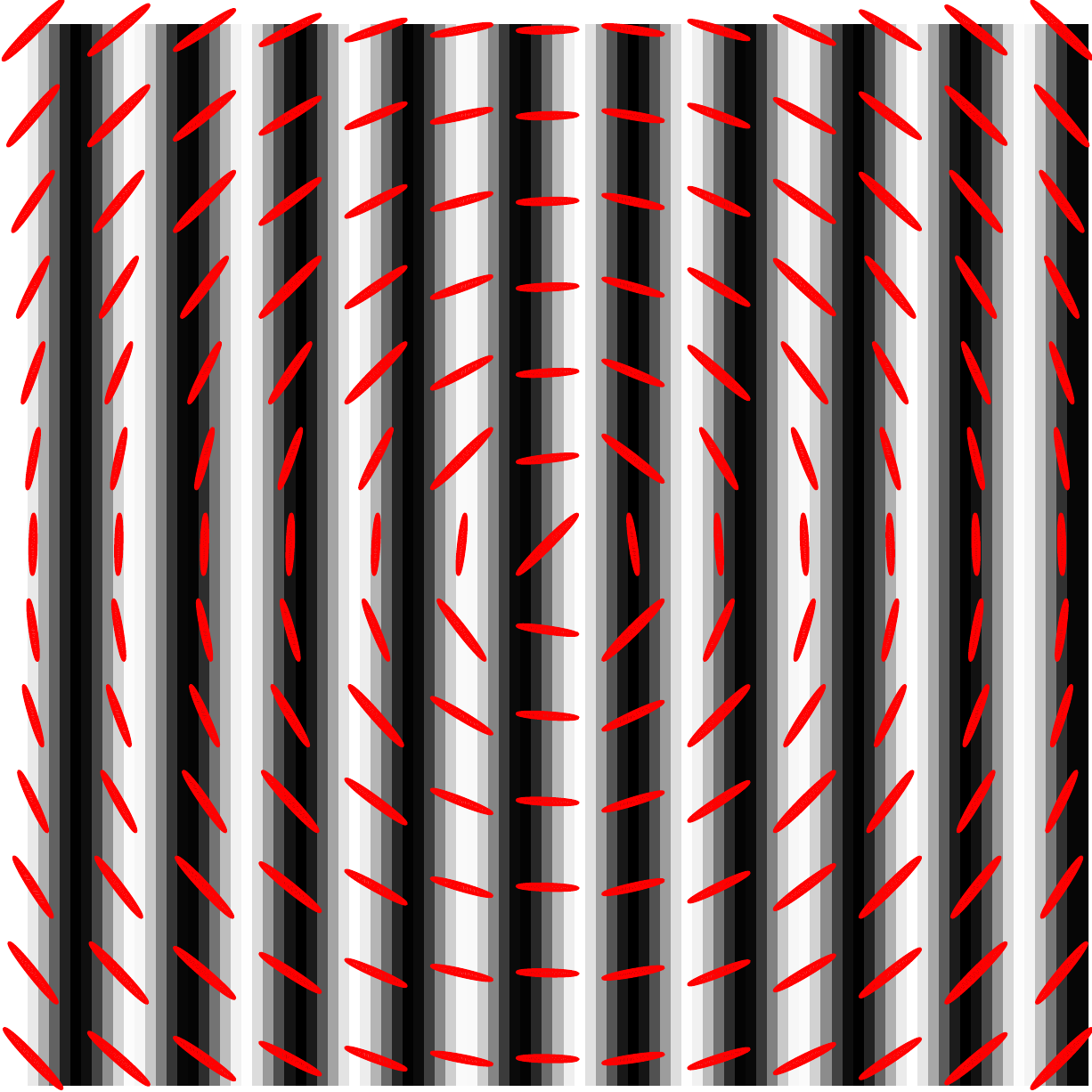} \\
		\myfig{aniso4}{1}&
		\myfig{aniso4}{5}&
		\myfig{aniso4}{10}&
		\myfig{aniso4}{15}&
		\myfig{aniso4}{20} &
		\myfigPot{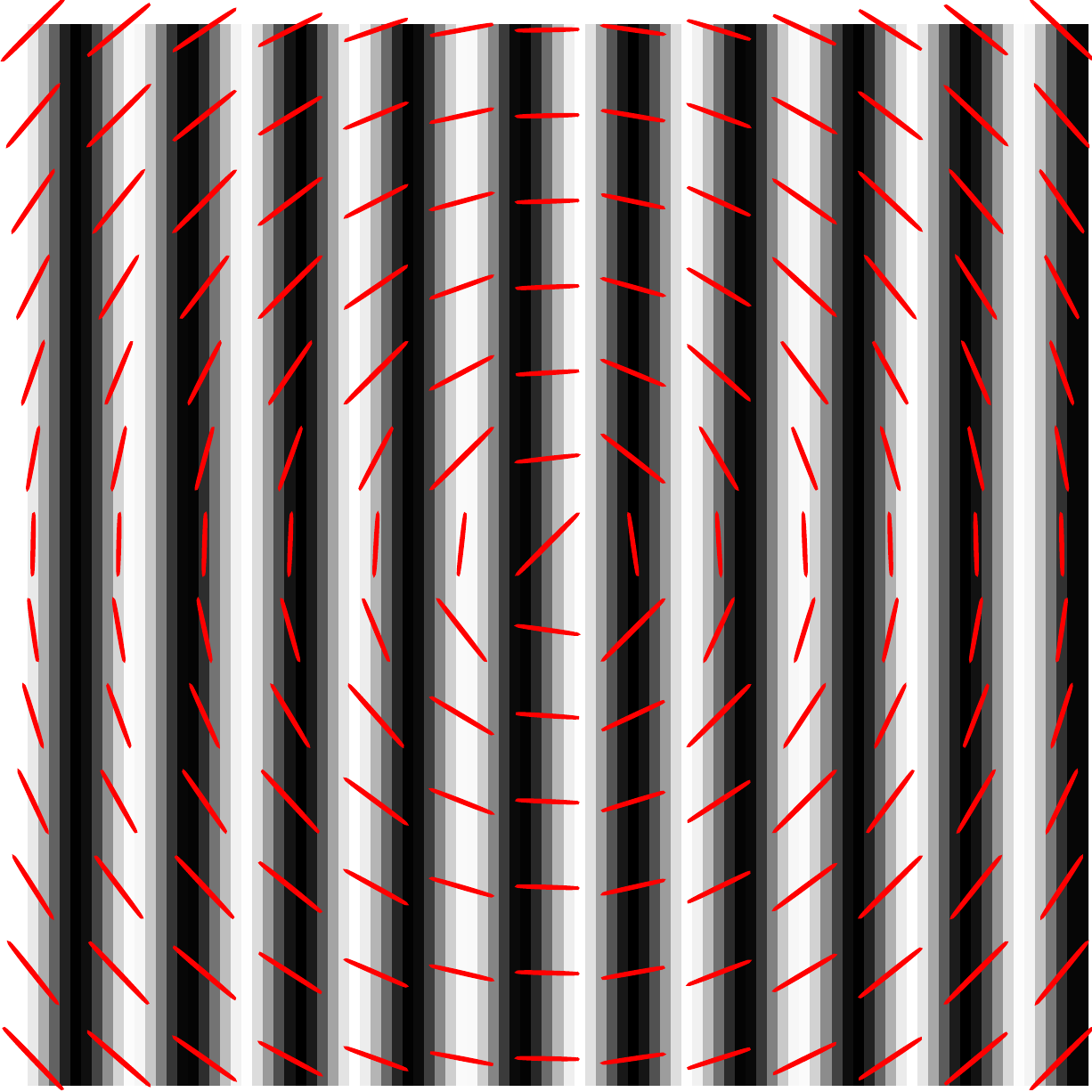} \\
		$t=0$ & $t=15$ & $t=30$ & $t=45$ & $t=60$ & $\cos(w), T$ 
	\end{tabular}
	\caption{%
		Display of the evolution $p_t$ using anisotropic diffusion kernels.
		The right column displays in the background the level-sets of $w$ and the tensor field $T(x)$
		displayed as red ellipsoids. 
		An ellipsoid at point $x$ is oriented along the principal axis of $T(x)$, and the length/width ratio is proportional to the anisotropy of $T(x)$. The metric thus favors mass movements along the direction of the ellipsoids.
		The anisotropy (ratio between maximum and minimum eigenvalues of $T(x)$)
		is set respectively from top to bottom to $\{2, 10, 30\}$ in each of the successive rows.
	}
   \label{fig-anisotropic}
\end{figure}

\subsection{Non-linear Diffusions}

To model non-linear diffusion equations, we consider (possibly space-varying) generalized entropies
\eql{\label{eq-defn-gen-entropies}
	f(p) \eqdef \sum_i b_i e_{m_i}(p_i)
	\qwhereq
	\foralls m \geq 1, 
	e_m(s) \eqdef \choice{
		s (\log(s)-1) \qifq m=1, \\
		s \frac{s^{m-1}-m}{m-1}	\qifq m>1.
	}
}
Here $(b_i)_{i=1}^N$ is a set of weights $b_i \geq 0$ that enable a specially varying diffusion strength, while $(m_i)_{i=1}^N$ is a set of exponents that enable to make the evolution more non-linear at certain locations. Note that the case $m=1$ corresponds to minus the entropy defined in~\eqref{eq-entropy-defn}.

In the case of constant weights $b$ and exponents $m$, the gradient flows of these functionals lead to non-linear diffusions of the form $\partial_t p = \Delta p^m$. The case $m=1$ is the usual linear heat diffusion, as considered in the initial work of~\cite{jordan1998variational}. The case $m=2$ is the so-called porous medium equation~\cite{otto2001geometry}, where diffusion is slower in areas where the density $p$ is small. In particular, solutions might have a compact support that evolves in time, on contrary to the linear heat diffusion where mass can travel with infinite speed.  

The following proposition, whose proof follows from writing the first order condition of~\eqref{eq-defn-proxKL}, details how to compute the proximal operator of $h$.

\begin{prop}
The proximal operator of $f$ satisfies 
\eq{
	\Prox_{\si f}^{\oKL}(r) = (\Prox_{\si b_i e_{m_i}}^{\oKL}(r_i) )_{i=1}^N.
}
For $m=1$, the proximal operator of $e_1$ reads
\eql{\label{eq-prox-entropy}
	\foralls s>0, \quad \Prox_{\si e_1}^{\oKL}(s) = s^{\frac{1}{1+\si}}.
}
If $m \neq 1$, then for any $s>0$, $\ProxKL_{\si e_m}(s) = \psi$ is the unique positive root of the equation
\eql{\label{eq-prox-psi}
	\log(\psi) + m \si \frac{\psi^{m-1} - 1}{m-1}  = \log(s)
}
\end{prop}

In the numerical applications, we compute $\ProxKL_{\si e_m}$ by using a few steps of Newton iterations to solve~\eqref{eq-prox-psi}, which can be parallelized over all the grid's locations. Figure~\ref{fig-Psi} shows examples of the energy $e_m$ and the corresponding proximal maps $\ProxKL_{\si e_m}$. They act as pointwise non-linear thresholding operators that are applied iteratively on the probability distribution being computed. In some sense, the congestion term~\eqref{eq-dfn-congestion} and the corresponding proximal operator~\eqref{eq-dfn-congestion-prox} can be understood as a limit of this model as $m \rightarrow +\infty$.

\newcommand{\myfigProx}[1]{\includegraphics[width=.32\linewidth]{prox/#1}}

\begin{figure}[h!]
	\centering
	\begin{tabular}{@{}c@{\hspace{1mm}}c@{\hspace{1mm}}c@{}} 
		\myfigProx{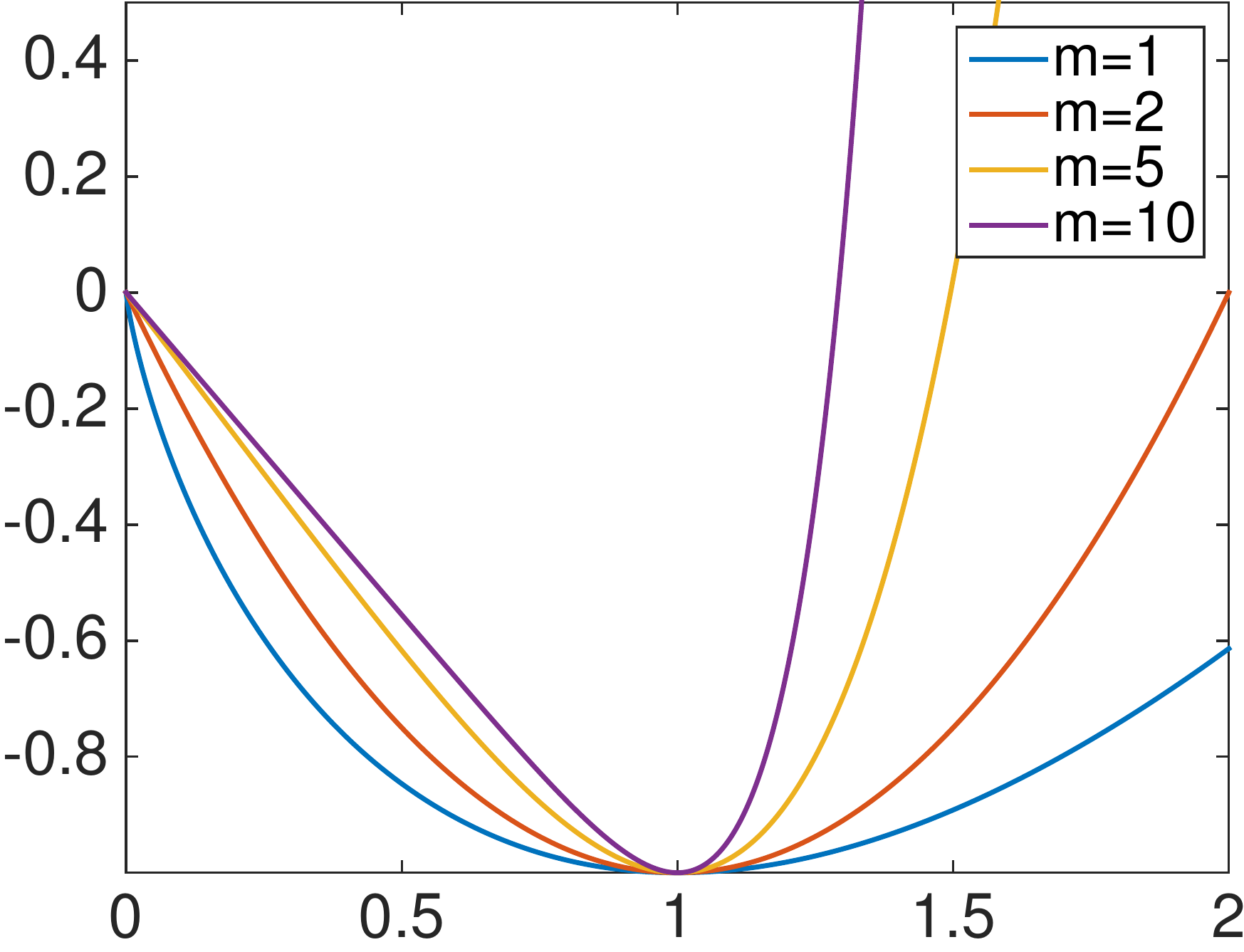} &
		\myfigProx{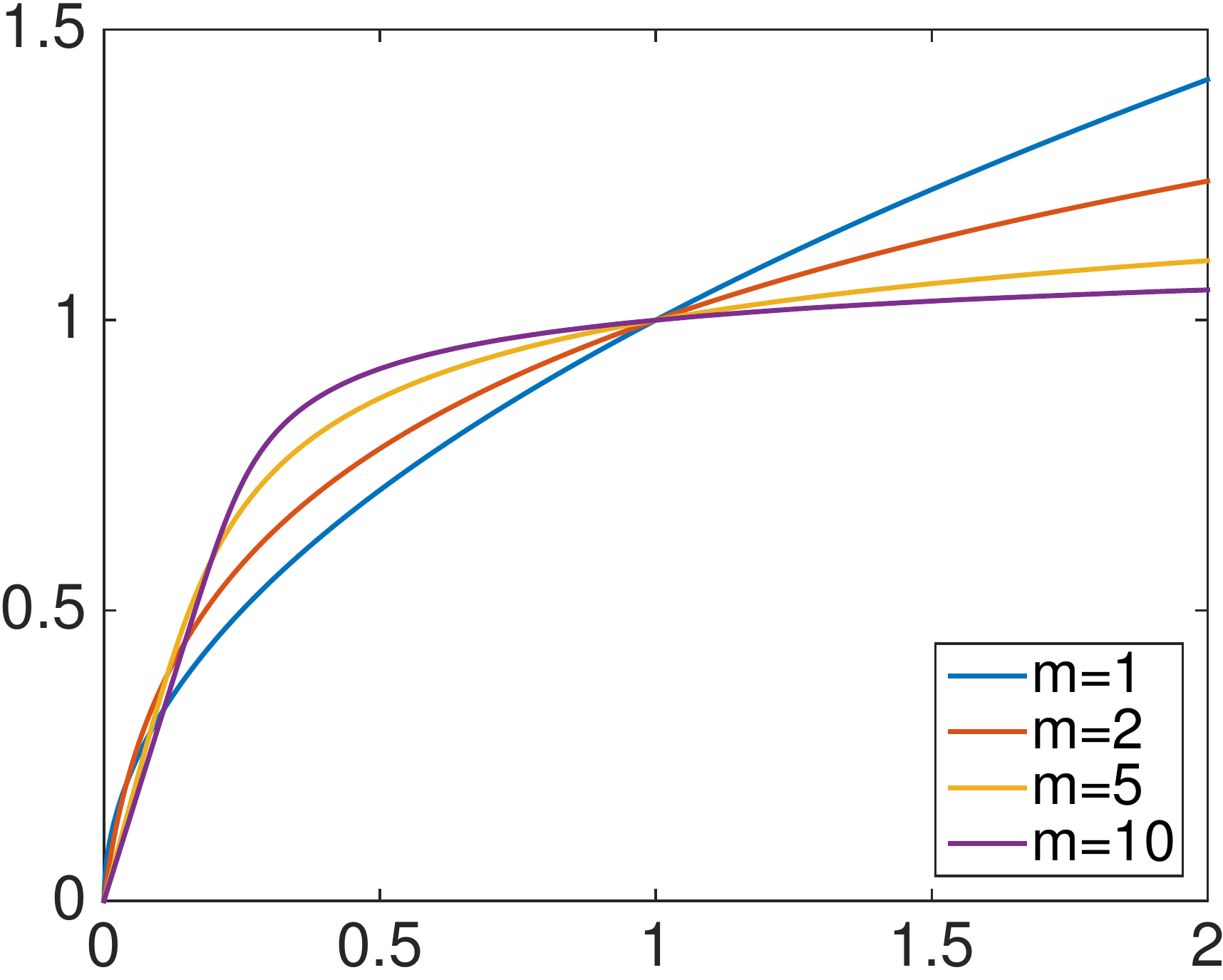} &
		\myfigProx{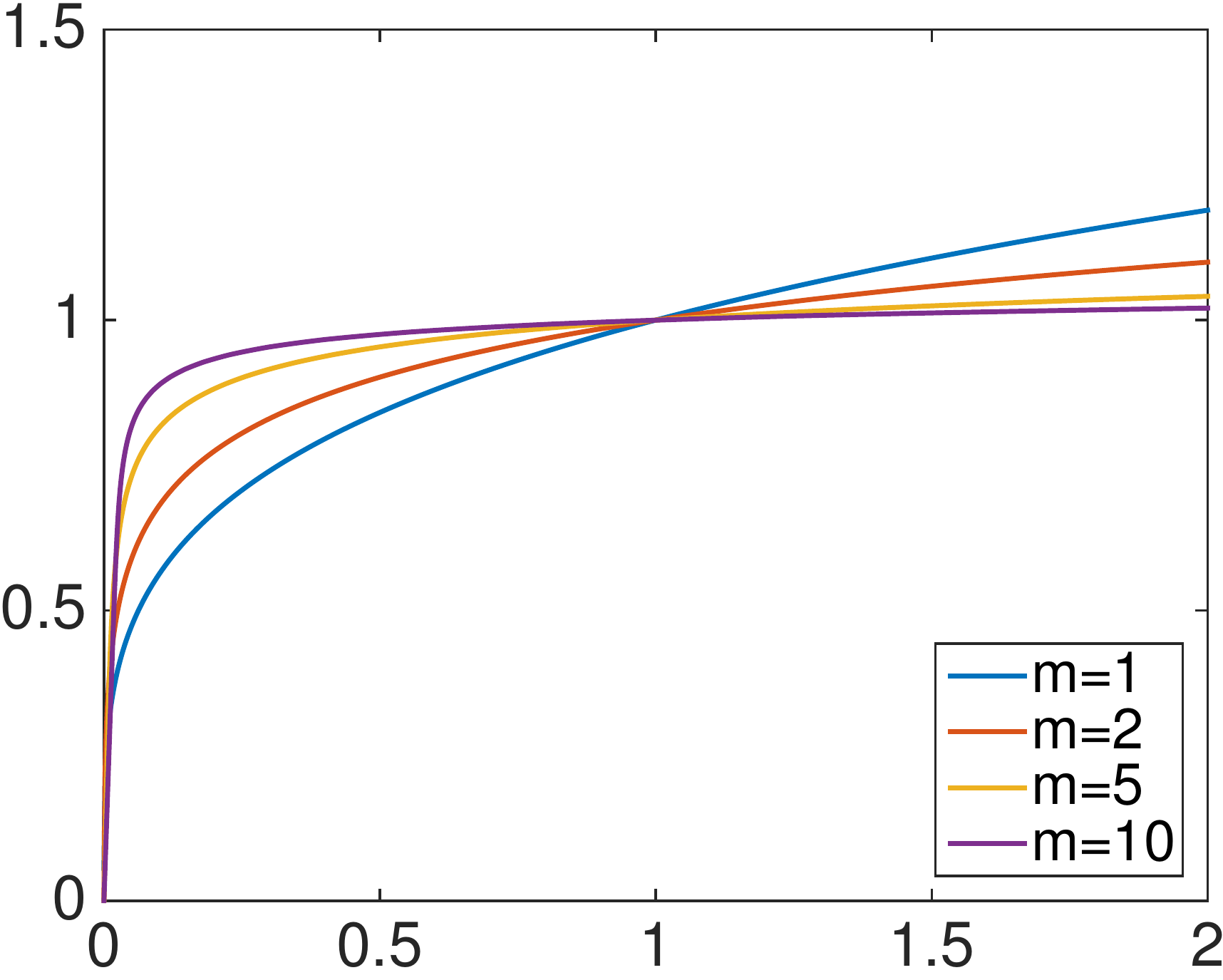} \\ 
		$e_m$ & $\ProxKL_{\si e_m}, \sigma=1$ & $\ProxKL_{\si e_m}, \sigma=3$ \\
	\end{tabular}
	\caption{%
		Display of the graphs of functions $e_m$ and $\ProxKL_{\si e_m}$ for some values of $(\si,m)$.
	}
   \label{fig-Psi}
\end{figure}

Figure~\eqref{fig-porous} shows illustration of the models in the case where either $b$ or $m$ is varying, thus producing a spatially varying flow. The initial distribution $p_{t=0}$ is computed as a white noise realization, where the pixels are independently and identically drawn according to a uniform distribution on $[0,1]$ (and then $p$ is normalized to unit mass). 

\newcommand{\myfigPor}[2]{\includegraphics[width=.19\linewidth]{rand-varying-#1/rand-varying-#1-#2}}

\begin{figure}[h!]
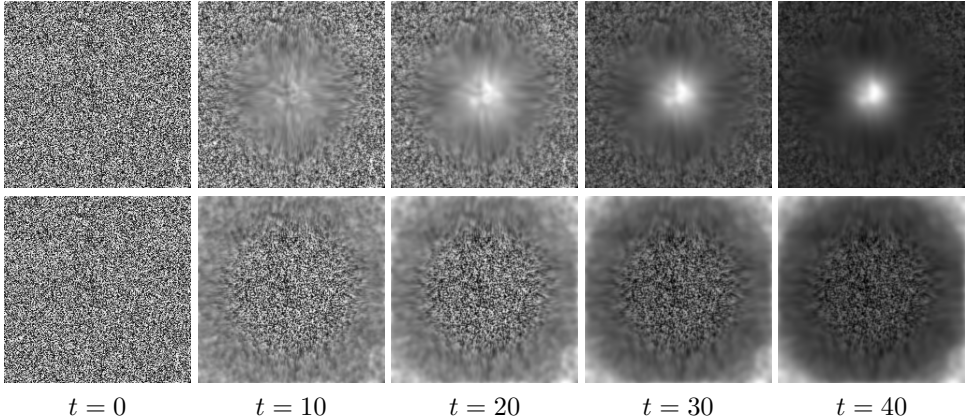

	\centering
	\begin{tabular}{@{}c@{\hspace{1mm}}c@{\hspace{1mm}}c@{\hspace{1mm}}c@{\hspace{1mm}}c@{}}
		\myfigPor{e}{1}&
		\myfigPor{e}{5}&
		\myfigPor{e}{10}&
		\myfigPor{e}{15}&
		\myfigPor{e}{20}\\
		\myfigPor{m}{1}&
		\myfigPor{m}{5}&
		\myfigPor{m}{10}&
		\myfigPor{m}{15}&
		\myfigPor{m}{20}\\
		$t=0$ & $t=10$ & $t=20$ & $t=30$ & $t=40$ 
	\end{tabular}
	\caption{%
		Non-linear diffusion by gradient flow on the generalized entropies~\eqref{eq-defn-gen-entropies}. 
		Top row: fixed $m_i=1.4$ and varying weights $b_i \in [1,20]$ (1 in the boundary, 20 in the center).
		Bottom row: fixed $b_i=1$ and varying exponents $m_i \in [1,2]$ (1 in the boundary, $2$ in the center).
	}
   \label{fig-porous}
\end{figure}

\subsection{Non-convex Functionals}

It is formally possible to apply Dykstra's algorithm detailed in Section~\ref{sec-kl-jko-algo} to a non-convex function $f$, if one is able to compute in closed form the proximal operator~\eqref{eq-defn-proxKL} (which then might be a multi-valued map). Of course there is no hope for the resulting non-convex Dykstra's algorithm to converge in general to the global minimizer of the non-convex optimization~\eqref{eq-defn-KL-prbm}. Even worse, to the best of our knowledge, there is currently no proof that the non-convex Dykstra's algorithm converges to a stationary point of the energy, even in the case of an Euclidean divergence. However, we found that applying Dykstra's algorithm to non-convex functions works remarkably well in practice. Note that the closely related Douglas-Rachford (DR) algorithm is known to converge in some particular non-convex cases~\cite{ArtachoNonCvx}. DR is known to perform very well on several non-convex optimization problems such as phase retrieval~\cite{BauschkeNonCvx}. 

To test this non-convex setting, we replace the congestion box constraint~\eqref{eq-dfn-congestion} by the non-convex function
\eql{\label{eq-dfn-congestion-noncvx}
	f(p) \eqdef \iota_{\{0,\kappa\}^N} + \dotp{w}{p}.
}
This function enforces that the thought after solution is binary, so that each value $p_i$ is in $\{0,\kappa\}$.
The proximal operator of this non-convex function can be computed explicitly using
\eq{
	\foralls i \in \{1,\ldots,N\}, \quad
	\Prox_{\si \iota_{\{0,\kappa\}^N}}^{\oKL}(p)_i = \choice{
		0 \qifq p_i < \kappa/e, \\
		\{0,\kappa\} \qifq p_i = \kappa/e, \\
		\kappa \qifq p_i > \kappa/e, 
	}
}
where $e=\exp(1)$. 
Note that $\Prox_{\si \iota_{\{0,\kappa\}^N}}^{\oKL}(p)_i$ is multi-valued at $p_i = \kappa/e$, and numerically one needs to chose one of the two values. 
Figure~\ref{fig-nonconvex} shows a comparison of the evolutions obtained with the convex and non-convex functionals. The non-convex one suffers from binary noise artefacts, which could be partly due to the non-convexity, but also to the amplification of discretization errors by the  proximal mapping which is not Lipschitz continuous. 

\renewcommand{\myfig}[2]{\includegraphics[width=.19\linewidth]{#1-kappa10/#1-kappa10-#2}}

\begin{figure}[h!]
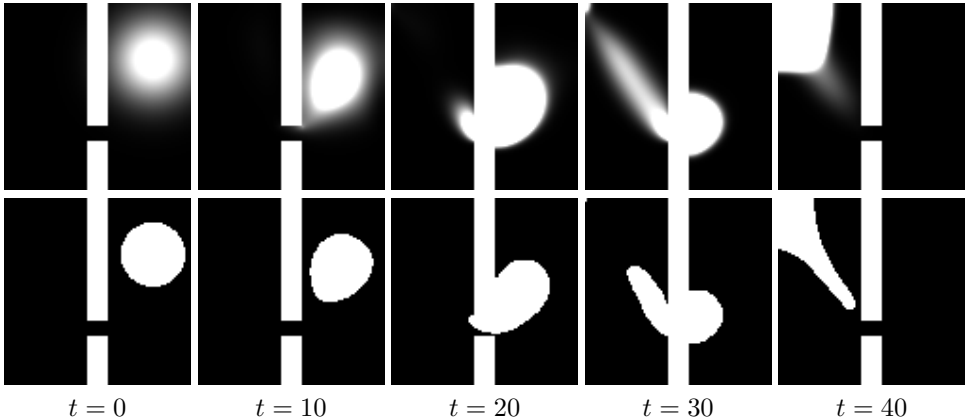

	\centering
	\begin{tabular}{@{}c@{\hspace{1mm}}c@{\hspace{1mm}}c@{\hspace{1mm}}c@{\hspace{1mm}}c@{}}
		\myfig{tworooms}{1}&
		\myfig{tworooms}{2}&
		\myfig{tworooms}{5}&
		\myfig{tworooms}{10}&
		\myfig{tworooms}{20}\\
		\myfig{tworooms-noncvx}{1}&
		\myfig{tworooms-noncvx}{2}&
		\myfig{tworooms-noncvx}{5}&
		\myfig{tworooms-noncvx}{10}&
		\myfig{tworooms-noncvx}{20}\\
		$t=0$ & $t=10$ & $t=20$ & $t=30$ & $t=40$
	\end{tabular}
	\caption{%
		Display of crowd evolution for $\kappa=\normi{p_{t=0}}$. 
		Top row: convex function~\eqref{eq-dfn-congestion}.
		Bottom row: non-convex function~\eqref{eq-dfn-congestion-noncvx}.
	}
   \label{fig-nonconvex}
\end{figure}


\section{More General Functionals}
\label{sec-general-functinons}

In order to highlight the power of the proposed entropic regularization approach, we show here how to adapt the algorithm detailed in Section~\ref{sec-bregman-prox} in order to deal with more involved functionals. These functionals require the introduction of several couplings, which in turn necessitates to develop a generic iterative scaling procedure derived from Dykstra's algorithm. This new method has its own interest, beyond the computation of Wasserstein gradient flows.

\subsection{A Generic Diagonal Scaling Algorithm}

In order to tackle a more general class of functions $f$, we consider here a generalization of problem~\eqref{eq-defn-KL-prbm} where one wants to optimize over a family $\pi=(\pi_1,\ldots,\pi_M)$ of $M$ couplings $\pi_m \in \RR^{N \times N}$ a functional of the form 
\eql{\label{eq-pbm-variational-generic}
	\umin{\pi \in (\RR^{N \times N})^M } \KL_\la(\pi|\xi) + \phi_1(\pi) + \phi_2(\pi)
	\qwhereq
	\choice{
		\phi_1(\pi) \eqdef \psi_1(\pi \ones), \\
		\phi_2(\pi) \eqdef \psi_2(\pi^T \ones), 
	}
}
where, for $\la \in \RR_+^M$, $\KL_\la$ is the weighted KL divergence (see also~\eqref{eq-defn-okl-lambda})
\eq{
	\foralls (\pi,\xi) \in (\RR^{N \times N})^M \times (\RR^{N \times N})^M, \quad
	\KL_\la(\pi|\xi) = \sum_{m=1}^M \la_m \KL(\pi_m|\xi_m)
}
and where we denoted, with a slight abuse of notations, the collection of left and right marginals as
\eq{
	\pi \ones = (\pi_1\ones,\ldots,\pi_M \ones) \in (\RR^N)^M
	\qandq
	\pi^T \ones = (\pi_1^T\ones,\ldots,\pi_M^T \ones) \in (\RR^N)^M
}
and $\psi_i : (\RR^N)^M \rightarrow \RR$ are convex functions for which one can compute the proximal operator $\Prox^{\oKL_\la}_{\psi_i}$ according to the $\oKL_\la$ divergence.

We wish to apply Dykstra's iterations~\eqref{eq-iter-dystra-1} and~\eqref{eq-iter-dystra-2} to~\eqref{eq-pbm-variational-generic}. This requires to compute the proximal operator of the functions $\phi_i$. The following proposition details how to achieve this using the proximal operator of the functions $\psi_i$ alone.

\begin{proposition}	
	We denote, for $i=1,2$, $\pi^{[i]} \eqdef \Prox^{\KL_\la}_{\phi_i}(\pi)$.
	We denote, for $m=1,\ldots,M$, $\tilde p_m^{[1]} \eqdef \pi_m \ones$ and $\tilde p_m^{[2]} \eqdef \pi_m^T \ones$.
	One has
	\begin{align*}
	\foralls m \in \{1,\ldots,M\}, \quad
		\pi_m^{[1]} &= \diag\pa{ \frac{p_m^{[1]}}{\tilde p_m^{[1]}} } \pi_m  \qandq
		\pi_m^{[2]} = \pi_m \diag\pa{ \frac{p_m^{[2]}}{\tilde p_m^{[2]}} } \\
		\qwhereq \foralls i \in \{1,2\}, \quad
		(p^{[i]})_m & = \Prox^{\oKL_\la}_{\psi_i}\pa{ (\tilde p^{[i]})_m }.
	\end{align*}
\end{proposition}
\begin{proof}
	This corresponds to an application of formulas~\eqref{eq-prox-calculus-lifting} and~\eqref{eq-prox-calculus-lifting-transp}.
\end{proof}

The following proposition, which is similar to Proposition~\ref{prop-implementation-scaling}, explains how to implement the iterations of Dykstra's algorithm using only multiplications with the kernels $(\xi_m)_m$. 

\begin{proposition}\label{prop-implementation-generic}
The iterates $\iter{\pi} = (\iter{\pi_1},\ldots,\iter{\pi_M})$ of Dykstra's algorithm can be written  as
\eq{
	\foralls m \in \{1,\ldots,M\}, \quad
	\iter{\pi_m} = \diag(\iter{a_m}) \xi_m \diag(\iter{b_m})
	\qandq
	\iter{z_m} = \iter{u_m} v_m^{(\ell),T}.
}
with the initialization 
\eq{
	\foralls m \in \{1,\ldots,M\}, \quad
	\itz{a_m}=\itz{b_m}=\itz{u_m}=\itz{v_m} \eqdef \ones.
}
We define, $\foralls m \in \{1,\ldots,M\}$, 
\begin{align*}
	\itA{\tilde a_m} &\eqdef \itA{a_m} \odot \itAA{u_m} \qandq
	\itA{\tilde b_m} \eqdef \itA{b_m} \odot \itAA{v_m}, \\
	(p_m)_m &\eqdef \Prox^{\oKL_\la}_{\psi_{[\ell]_2}}( (\tilde p_m)_m )
	\qwhereq
	\tilde p_m \eqdef \choice{
			\itA{\tilde a_m} \odot \xi_m( \itA{\tilde b_m} ) \qifq [\ell]_2=1, \\
			\itA{\tilde b_m} \odot \xi_m^T( \itA{\tilde a_m} ) \qifq [\ell]_2=2.
	}		
\end{align*}
The update reads, $\foralls m \in \{1,\ldots,M\}$, 
\begin{align*}
	\iter{a_m} &\eqdef \choice{
			p_m \odot [\xi_m(\itA{\tilde b_m})]^{-1} \qifq [\ell]_2=1 \\
			\itA{\tilde a_m} \qifq [\ell]_2=2
		} \\
	\iter{b_m} &\eqdef \choice{
			\itA{\tilde b_m} \qifq [\ell]_2=1 \\
			p_m \odot [\xi_m^T(\itA{\tilde a_m})]^{-1} \qifq [\ell]_2=2
		} \\
	\iter{u_m}  &\eqdef 	\itAA{u_m} \odot \frac{ \itA{a_m} }{ \iter{a_m} }
	\qandq
	\iter{v_m} \eqdef \itAA{v_m} \odot \frac{ \itA{b_m} }{ \iter{b_m} }.
\end{align*}
\end{proposition}

\subsection{Wasserstein Attraction with Congestion}

We now give a first concrete example of functional $f$ for which the formulation~\eqref{eq-pbm-variational-generic} should be used in place of~\eqref{eq-defn-KL-prbm}. 

Instead of advecting the mass of $p_t$ according to a fixed potential $w$ as it is considered in the functional~\eqref{eq-dfn-congestion}, it is possible to make it evolve toward a ``target'' distribution $r \in \Si_N$ by minimizing the Wasserstein distance between $p_t$ and $r$. It thus consists in considering the gradient flow of the function 
\eql{\label{eq-def-f-attrac}
	\foralls p \in \Si_N, \quad
	f(p) = W_\ga(r,p) + h(p) + \iota_{\Si_N}(p), 
} 
where $h(p)$ is a function for which one can compute its $\oKL$ proximal operator as defined in~\eqref{eq-def-kl-prox-operator}. 

We now denote $q \eqdef p_t$ the previous iterate, and aim at solving a single JKO step~\eqref{eq-prox-step}. It is not possible to compute in closed form the $\oKL$ proximal operator of the function $f$ defined in~\eqref{eq-def-f-attrac}, so that the algorithm detailed in Section~\ref{sec-bregman-prox} is not directly applicable. 

Instead, we re-formulate~\eqref{eq-prox-step} as a KL minimization of the form~\eqref{eq-pbm-variational-generic} involving $M=2$ couplings $\pi=(\pi_1,\pi_2) \in (\RR^{N \times N})^2$ and kernels $\xi=(\xi_1,\xi_2) \eqdef (e^{-c/\ga},e^{-c/\ga})$. This encodes implicitly the solution $p=\pi_1\ones=\pi_2 \ones$ of~\eqref{eq-prox-step} using the solution $\pi=(\pi_1,\pi_2)$ of~\eqref{eq-pbm-variational-generic} when introducing the functions
\begin{align*}
	\psi_1(p_1,p_2) &= \iota_{\Dd}(p_1,p_2) + h(p_1) 
	\qwhereq
	\Dd = \enscond{(p_1,p_2)}{p_1=p_2}, \\
	\psi_2(p_1,p_2) &= \iota_{\{q,r\}}(p_1,p_2), 
\end{align*}
and the weights $\la=(1,\tau) \in \RR_+^2$. The following proposition details how to compute the proximal operator of these functionals. It is important to remind that that these functionals as well as their respective proximal operators operate on vectors of $\RR^N$, not on couplings.

\begin{proposition}
	One has
	\begin{align}\label{eq-prox-psi-attraction}
		\Prox^{\oKL_\la}_{\psi_1}(p_1,p_2) &= (p,p)
		\qwhereq
		p=  \Prox^{\oKL}_{\frac{1}{1+\tau}h}\pa{ p_1^{\frac{1}{1+\tau}} \odot p_2^{\frac{\tau}{1+\tau}} } \\
		\Prox^{\oKL_\la}_{\psi_2}(p_1,p_2) &= (q,r) 
	\end{align}
\end{proposition}
\begin{proof}
	The computation of $\Prox^{\oKL_\la}_{\psi_1}$ follows from~\eqref{eq-prox-calculus-equal}.
	The computation of $\Prox^{\oKL_\la}_{\psi_2}$ follows from~\eqref{eq-prox-calculus-indic}.
\end{proof}

With these proximal maps at hands, and with the formula for the iterations detailed in Proposition~\ref{prop-implementation-generic}, one can solve for each JKO step by computing the optimal $(\pi_1,\pi_2)$ using $q \eqdef p_t$ and then updating $p_{t+1} \eqdef \pi_1 \ones = \pi_2 \ones$.

\paragraph{Numerical Illustrations}

In order to introduce some congestion, we consider here the function $h(p)=\iota_{[0,\kappa]^N}$, as in~\eqref{eq-dfn-congestion}, and its KL proximal operator is computed as detailed in~\eqref{eq-dfn-congestion-prox}.

Figure~\ref{fig-wass-attrac} shows some examples of such a JKO flows computed on a rectangular grid of $N = 100 \times 100$ points. The right hand side column shows the target distribution $r$. Note that the flow $p_t$ typically does not converge toward $r$ as $t \rightarrow +\infty$, because of the congestion effect.

\newcommand{\myfigW}[2]{\includegraphics[width=.16\linewidth]{wasserstein-attraction/#1/#1-#2}}

\begin{figure}[h!]
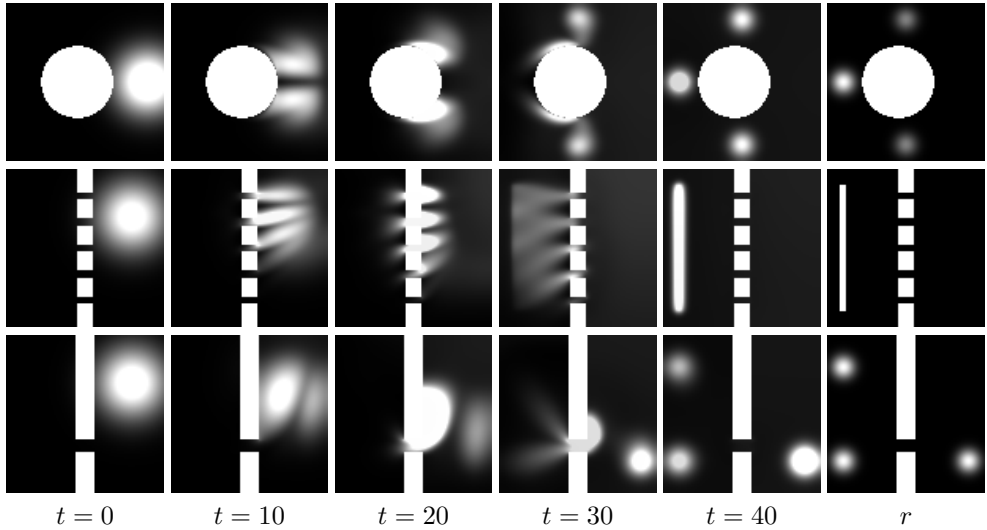

	\centering
	\begin{tabular}{@{}c@{\hspace{1mm}}c@{\hspace{1mm}}c@{\hspace{1mm}}c@{\hspace{1mm}}c@{\hspace{1mm}}c@{}}
		\myfigW{disk}{1}&
		\myfigW{disk}{2}&
		\myfigW{disk}{5}&
		\myfigW{disk}{10}&
		\myfigW{disk}{20}&
		\myfigW{disk}{tgt}\\
		\myfigW{holes}{1}&
		\myfigW{holes}{2}&
		\myfigW{holes}{5}&
		\myfigW{holes}{10}&
		\myfigW{holes}{20}&
		\myfigW{holes}{tgt}\\
		\myfigW{tworooms}{1}&
		\myfigW{tworooms}{2}&
		\myfigW{tworooms}{5}&
		\myfigW{tworooms}{10}&
		\myfigW{tworooms}{20}&
		\myfigW{tworooms}{tgt}\\
		$t=0$ & $t=10$ & $t=20$ & $t=30$ & $t=40$  & $r$ 
	\end{tabular}
	\caption{%
		Examples of gradient flows for the Wasserstein attraction toward the density $r$ displayed on the right column. The congestion parameter is set to $\kappa=\normi{p_{t=0}}$.
	}
   \label{fig-wass-attrac}
\end{figure}

\subsection{Multiple Densities Evolutions}
\label{sec-multiple-densities}

A natural extension of the JKO flow~\eqref{eq-smooth-jko} is to describe the evolution of a finite family of densities $p_t=(p_{i,t})_i$ by minimizing a function $f( (p_i)_{i} )$, where one defines the transport distance as the sum of independent Wasserstein distances
\eq{
	W_\ga( (p_i)_i, (q_i)_i ) \eqdef \sum_{i} W_\ga(p_i,q_i).
}
The function $f$ thus introduce a coupling between densities during the evolution.
For simplicity we consider in the following the case of $2$ densities. 

\paragraph{Wasserstein pairwise attraction}

We first consider the case where the coupling is a Wasserstein attraction between the two densities
\eq{
	f(p_1,p_2) = \al W(p_1,p_2) + h_1(p_1) + h_1(p_2)
}
where the functions $h_i$ are ``simple'' so that one can compute easily $\Prox_{h_i}^{\oKL}$.

Denoting $q=(q_1,q_2) \eqdef (p_{1,t},p_{2,t})$ the previous iterate at time $t$, the solution $p=p_{t+1}$ to the JKO step~\eqref{eq-prox-step} can be written as 
\eq{
	p = (p_1,p_2) = (\pi_1 \ones, \pi_2\ones) = (\pi_3^T \ones,\pi_3^T\ones), 
}
where one needs to solve for $M=3$ couplings $(\pi_1,\pi_2,\pi_3)$ the problem~\eqref{eq-pbm-variational-generic} with the functionals
\begin{align*}
	\psi_1(a_1,a_2,a_3) &= 
		\iota_{\Dd}(a_1,a_3) + 
		\iota_{\{q_2\}}(a_2) + 
		\frac{\tau}{\ga} h_1(a_1), \\
	\psi_2(a_1,a_2,a_3) &= 
		\iota_{\Dd}(a_2,a_3) + 
		\iota_{\{q_1\}}(a_1) + 
		\frac{\tau}{\ga} h_2(a_2)
\end{align*}
with KL weights $\la=(1,1,\tau\al)$. The proximal operators of these functions are easy to compute as detailed in the following proposition.

\begin{prop}
	For $i \in \{1,2\}$, denoting $j=3-i \in \{1,2\}$, one has
	\eq{
		(b_m)_m = \Prox_{\psi_i}^{\oKL_\la}(a_m)_m
		\qwhereq
		b_i=b_3=\Prox_{h_i}( a_i^{\frac{1}{1+\tau\al}} \odot a_3^{\frac{\tau}{1+\tau\al}} ), \quad
		b_j = q_j.
	}
\end{prop}
\begin{proof}
	The expression for $(b_i,b_3)$ is obtained by using~\eqref{eq-prox-calculus-equal}.
	The expression for $b_j$ is obtained by using~\eqref{eq-prox-calculus-indic}.
\end{proof}

In the numerical example, we used $h_i(p_i) \eqdef \iota_{[0,\kappa]^N}(p_i)+\dotp{w_i}{p_i}$ for potentials $(w_1,w_2)  \in \RR^N \times \RR^N$. The $\KL$ proximal operator of these functions can be computed as detailed in Proposition~\ref{prop-prox-congest}. Figure~\ref{fig-pairwise-attraction} displays the results obtained on a rectangular grid of $N=200 \times 200$ points.

\newcommand{\myfigPairA}[3]{\includegraphics[width=.19\linewidth]{pairwise-attraction/#1-#2/#1-#3}}

\begin{figure}[h!]
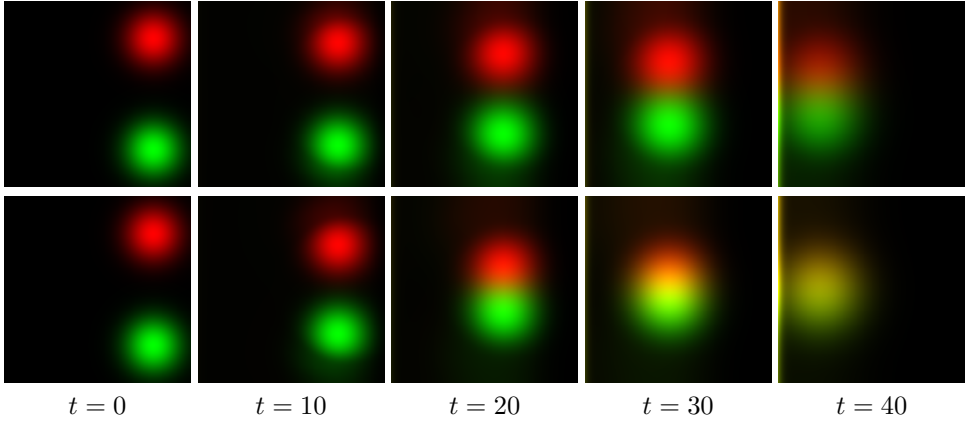

	\centering
	\begin{tabular}{@{}c@{\hspace{1mm}}c@{\hspace{1mm}}c@{\hspace{1mm}}c@{\hspace{1mm}}c@{}}
		\myfigPairA{twobumps}{10}{1}&
		\myfigPairA{twobumps}{10}{2}&
		\myfigPairA{twobumps}{10}{5}&
		\myfigPairA{twobumps}{10}{8}&
		\myfigPairA{twobumps}{10}{12}\\
		\myfigPairA{twobumps}{30}{1}&
		\myfigPairA{twobumps}{30}{2}&
		\myfigPairA{twobumps}{30}{5}&
		\myfigPairA{twobumps}{30}{8}&
		\myfigPairA{twobumps}{30}{12}\\
		$t=0$ & $t=10$ & $t=20$ & $t=30$ & $t=40$ 
	\end{tabular}
	\caption{%
		Evolution with a pairwise attraction between two densities, with congestion parameter $\kappa=\normi{p_{1,t=0}}=\normi{p_{2,t=0}}$.
		Display of both $p_{1,t}$ (red) and $p_{2,t}$ (green), yellow indicates a mixing.
		Top row: $\al=1$. Bottom row: $\al=3$.
	}
   \label{fig-pairwise-attraction}
\end{figure}

\paragraph{Summation couplings}

Another way to introduce some interaction between $p_1$ and $p_2$ is to consider a coupling on the sum
\eq{
	f(p_1,p_2) \eqdef h(p_1+p_2) + \dotp{p_1}{w_1} + \dotp{p_2}{w_2}.
}
for some function $h$ for which one can compute easily $\Prox_{h}^{\KL}$.

In this case, the solution $p=p_{t+1}$ to the JKO step~\eqref{eq-prox-step} can be written as $p=(p_1,p_2) = (\pi_1 \ones, \pi_2\ones)$ where the $M=2$ couplings $(\pi_1,\pi_2)$ solves problem~\eqref{eq-pbm-variational-generic} with the functionals
\begin{align*}
	\psi_1(a_1,a_2) &= 
		\frac{\tau}{\ga} f(a_1,a_2), \\
	\psi_2(a_1,a_2) &= 
		\iota_{\{q_1,q_2\}}(a_1,a_2)
\end{align*}
and weights $\la=(1,1)$.
The proximal operators of the functions are easy to compute as detailed in the following proposition.

\begin{prop}
	One has
	\begin{align*}
		\Prox_{\psi_1}^{\oKL_\la}(a_1,a_2) &= 
			\frac{ \Prox_{\frac{\tau}{\ga} h}(\tilde a_1+\tilde a_2) }{\tilde a_1+\tilde a_2} \odot ( \tilde a_1,\tilde a_2 ) \\
		\Prox_{\psi_2}^{\oKL_\la}(a_1,a_2) &= (q_1,q_2).
	\end{align*}
	where $\tilde a_i = a_i \odot e^{-\frac{\tau}{\ga} w_i}$
\end{prop}
\begin{proof}
	The expression for $\Prox_{\psi_1}^{\oKL_\la}$ is obtained by combining~\eqref{eq-prox-calculus-shift} and~\eqref{eq-prox-calculus-lifting}.
	The expression for $\Prox_{\psi_2}^{\oKL_\la}$ is obtained by using~\eqref{eq-prox-calculus-indic}.
\end{proof}

\newcommand{\myfigPair}[2]{\includegraphics[width=.19\linewidth]{pairwise-sum/#1/#1-#2}}

\begin{figure}[h!]
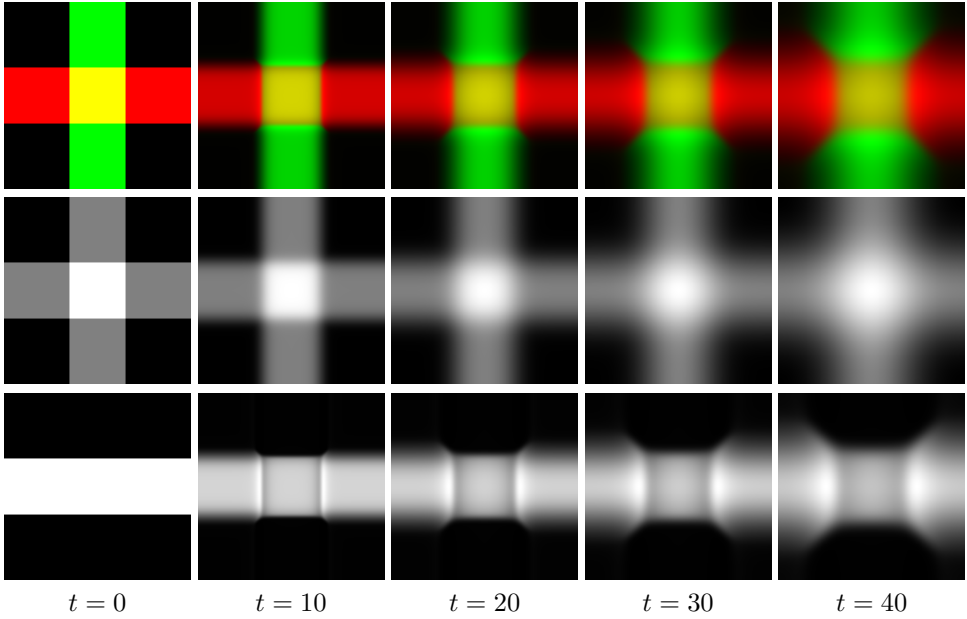

	\centering
	\begin{tabular}{@{}c@{\hspace{1mm}}c@{\hspace{1mm}}c@{\hspace{1mm}}c@{\hspace{1mm}}c@{}}
		\myfigPair{tworectangles}{1}&
		\myfigPair{tworectangles}{2}&
		\myfigPair{tworectangles}{5}&
		\myfigPair{tworectangles}{10}&
		\myfigPair{tworectangles}{20}\\
		\myfigPair{tworectangles}{sum-1}&
		\myfigPair{tworectangles}{sum-2}&
		\myfigPair{tworectangles}{sum-5}&
		\myfigPair{tworectangles}{sum-10}&
		\myfigPair{tworectangles}{sum-20}\\
		\myfigPair{tworectangles}{p1-1}&
		\myfigPair{tworectangles}{p1-2}&
		\myfigPair{tworectangles}{p1-5}&
		\myfigPair{tworectangles}{p1-10}&
		\myfigPair{tworectangles}{p1-20}\\
		$t=0$ & $t=10$ & $t=20$ & $t=30$ & $t=40$ 
	\end{tabular}
	\caption{%
		Evolution with a summation coupling $E(p_1+p_2)$.
		Top row: display of both $p_1$ (red) and $p_2$ (green), yellow indicates a mixing.
		Middle row: display of $p_1+p_2$, which evolves according to a linear heat diffusion.
		Bottom row: display of $p_1$.
	}
   \label{fig-pairwise-sum-entropy}
\end{figure}

As a first example, we consider an entropic coupling $h=E$, with $w_1=w_2=0$. Its proximal operator is computed in~\eqref{eq-prox-entropy}. A formal computation shows that, for the Euclidean $W_2$ transport on $\RR^d$, the corresponding discrete JKO steps~\eqref{eq-smooth-jko} is intended at approximating the non-linear PDE over $p_t = (p_{1,t},p_{2,t})$
\eq{
	\foralls i \in \{1,2\}, \quad
	\foralls t>0, \quad
	\partial_t p_{i,t} = \diverg\pa{ \frac{p_{i,t}}{p_{1,t} + p_{2,t}} \nabla p_{i,t} }.
}
This shows that while $p_t$ follows a non-linear coupled diffusion, $p_{1,t}+p_{2,t}$ follows a linear heat diffusion.
Figure~\ref{fig-pairwise-sum-entropy} shows a numerical illustration on a regular grid of $N=200 \times 200$ points.

As a second example, we consider a congestion coupling $h=\iota_{[0,\kappa]^N}$. Its proximal operator is computed in~\eqref{eq-dfn-congestion-prox}. Figure~\ref{fig-pairwise-sum-congestion} shows a numerical illustration  on a regular grid of $N=200 \times 200$ points. It shows two densities, initially supported on non-overlapping squares, moving in opposite directions under potentials $(w_1,w_2)$ such that $\nabla w_1=(1,0)^T$ and $\nabla w_2=(-1,0)^T$ (constant horizontal gradients). A congestion shock is created by the overlap of the densities, which in turn forces the support of the densities to be deformed and vertically enlarged.

\begin{figure}[h!]
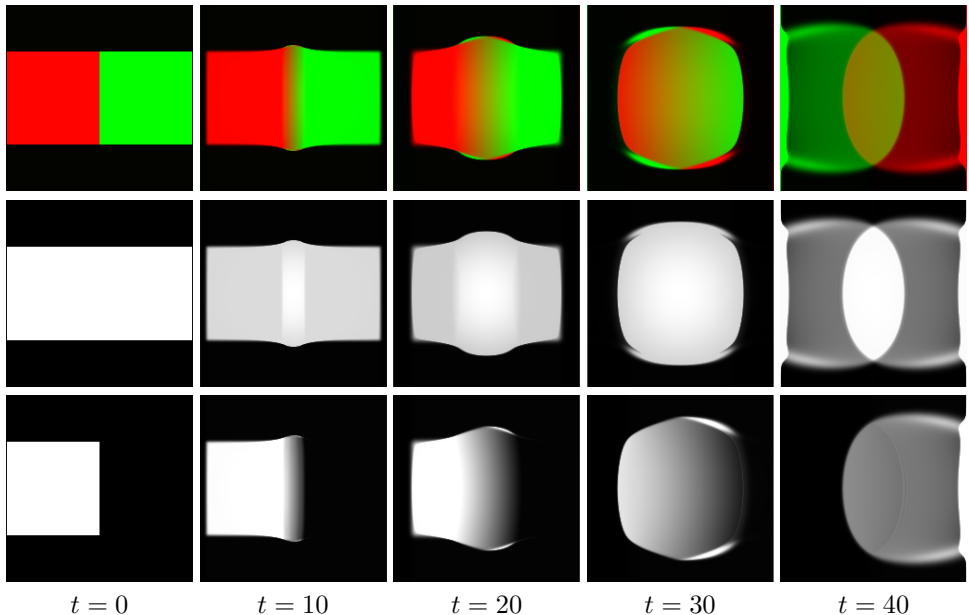

	\centering
	\begin{tabular}{@{}c@{\hspace{1mm}}c@{\hspace{1mm}}c@{\hspace{1mm}}c@{\hspace{1mm}}c@{}}
		\myfigPair{congestion}{1}&
		\myfigPair{congestion}{2}&
		\myfigPair{congestion}{5}&
		\myfigPair{congestion}{10}&
		\myfigPair{congestion}{20}\\
		\myfigPair{congestion}{sum-1}&
		\myfigPair{congestion}{sum-2}&
		\myfigPair{congestion}{sum-5}&
		\myfigPair{congestion}{sum-10}&
		\myfigPair{congestion}{sum-20}\\
		\myfigPair{congestion}{p1-1}&
		\myfigPair{congestion}{p1-2}&
		\myfigPair{congestion}{p1-5}&
		\myfigPair{congestion}{p1-10}&
		\myfigPair{congestion}{p1-20}\\
		$t=0$ & $t=10$ & $t=20$ & $t=30$ & $t=40$ 
	\end{tabular}
	\caption{%
		Evolution with a summation coupling $\iota_{[0,\kappa]^N}(p_1+p_2)$
		where $\kappa=\normi{p_{1,t=0}}=\normi{p_{2,t=0}}$.
		Top row: display of both $p_{1,t}$ (red) and $p_{2,t}$ (green), yellow indicates a mixing.
		Middle row: display of $p_{1,t}+p_{2,t}$.
		Bottom row: display of $p_{1,t}$.
	}
   \label{fig-pairwise-sum-congestion}
\end{figure}

\section*{Discussion and Conclusion}
\label{sec-conclusion}

In this paper, we have presented a novel algorithm to compute approximate discrete gradient flows according to an entropic smoothing of the Wasserstein distance. The main interest of the method is its speed, simplicity and versatility. This is achieved because the iterations only require (beside pointwise multiplications, divisions and exponentiations) to compute the successive applications of a ``convolution-like'' operator corresponding to the Gibbs kernel associated to the metric.

A natural question is to explore whether the discrete flow defined by~\eqref{eq-smooth-jko} has a continuous limit when $\tau_t=\tau \rightarrow 0$. If one uses a fixed $\ga_t=\ga>0$, this is not the case, because $W_\ga$ does not satisfies $W_\ga(p,p)=0$. More precisely, one has that 
\eq{
	\uargmin{q} W_\ga(p,q) = \xi \pa{ \frac{p}{\xi^T(p)} }, 
} 
so that the limit for small $\tau$ of $p_{t+1}$ defined by~\eqref{eq-smooth-jko} is a blurred (i.e. multiplied by $\xi$) version of $p_t$. Instead of using a fixed value for $\ga_t$, choosing $\ga_t=\ga(\tau_t)$ for some carefully chosen function of $\tau_t$ could allow the discrete flow to converge to the usual Wasserstein flow. We leave the analysis of this asymptotic setting to a future work.

\section*{Aknowledgements}

This work has been supported by the European Research Council (ERC project SIGMA-Vision). I would like to acknowledge  stimulating discussions with Marco Cuturi, Justin Solomon, Jean-David Benamou, Guillaume Carlier and Quentin Merigot.
I would like to thank Guillaume Carlier for suggesting to apply the method to multiple densities (Section~\ref{sec-multiple-densities}). I would like to thank Jean-Marie Mirebeau for giving me access to his code for anisotropic diffusion. I would like to  thank Antonin Chambolle and Jalal Fadili for suggesting me the proof strategy of Proposition~\ref{prop-conv-dykstra}. 
\appendix
\section{KL Proximal Calculus}
\label{sec-kl-calculus}

The following proposition details some useful property of the $\oKL$ proximal operator~\eqref{eq-def-kl-prox-operator}. This enables a powerful ``proximal calculus'' by combining these rules, which eases and simplifies the implementation of the algorithms. Note that we also consider generalized KL divergence over sets $(p_1,\ldots,p_M)$ of $M$ densities according to some weight $\la \in \RR_+^M$
\eql{\label{eq-defn-okl-lambda}
	\foralls (p_m)_m, (q_m)_m, \quad
	\oKL_\la( (p_m)_m | (q_m)_m ) \eqdef \sum_{m=1}^M \la_m \oKL(p_m|q_m).
}

\newcommand{\proxrule}[1]{}

\begin{proposition}
		\proxrule{Singleton constraint} For $f(p_1,\ldots,p_M) \eqdef \iota_{\{(q_1,\ldots,q_M)\}}(p_1,\ldots,p_M)$, one has
			\eql{\label{eq-prox-calculus-indic}
				\Prox_{f}^{\oKL_\la}(p_1,\ldots,p_M)=(q_1,\ldots,q_M).
			}
		\proxrule{Linear shift} For $f(p_1,\ldots,p_M) \eqdef h(p_1,\ldots,p_M)+\sum_{i=1}^M \dotp{w_i}{p_i}$, one has
			\eql{\label{eq-prox-calculus-shift}
				\Prox_{f}^{\oKL_\la}(p_1,\ldots,p_M) = 
				\Prox_{h}^{\oKL_\la}(p_1 \odot e^{-w_1/\la_1}, \ldots, p_M \odot e^{-w_M/\la_M}).
			}
		\proxrule{Equality constraint} For $f(p_1,\ldots,p_M) \eqdef \iota_\Dd(p_1,\ldots,p_M)+h(p_1,\ldots,p_M)$ where
			\eq{
				\Dd \eqdef \enscond{(p_1,\ldots,p_M)}{p_1=\ldots=p_M},
			} 
			one has
			\eql{\label{eq-prox-calculus-equal}
				\Prox_{f}^{\oKL_\la}(p_1,\ldots,p_M)= (p,\ldots,p)
				\!\qwhereq\!
				p = \Prox_{\frac{1}{\sum_i \la_i} \tilde h}^{\KL}\pa{ p_1^{\tilde\la_1} \odot \ldots \odot p_M^{\tilde\la_M} }, 	
			}
			where we denoted $\tilde\la_i \eqdef \la_i/\sum_{j} \la_j$ and $\tilde h(p)=h(p,\ldots,p)$. \\
		\proxrule{Composition with sum} For  $f(p_1,\ldots,p_M) \eqdef h(p_1+\ldots+p_M)$ and 
			$\la \eqdef (1,\ldots,1)$, one has 
			\eql{\label{eq-prox-calculus-sum}
				\Prox_{f}^{\oKL_\la}(p_1,\ldots,p_M)= \frac{\Prox_{h}^{\oKL}(p_1+\ldots+p_M)}{p_1+\ldots+p_M} (p_1,\ldots,p_M)
			}
		\proxrule{Lifting} We define $f(\pi_1,\ldots,\pi_M) \eqdef h(\pi_1\ones,\ldots,\pi_M\ones)$. 
			We denote 
			\eq{
				\foralls m \in \{1,\ldots,M\}, \quad p_m \eqdef \pi_m\ones
				\qandq
				(\tilde p_1,\ldots,\tilde p_M) \eqdef  \Prox_h^{\oKL_\la}( p_1,\ldots,p_M ).
			}
			One has
			\eql{\label{eq-prox-calculus-lifting}
				\Prox_{f}^{\KL_\la}(\pi_1,\ldots,\pi_M) =
				\pa{
					\diag\pa{ \frac{\tilde p_m}{p_m} } \pi
				}_m
			}
		\proxrule{Lifting, transposed} We define $f(\pi_1,\ldots,\pi_M) \eqdef h(\pi_1^T\ones,\ldots,\pi_M^T\ones)$. 
			We denote 
			\eq{
				\foralls m \in \{1,\ldots,M\}, \quad p_m \eqdef \pi_m^T \ones
				\qandq
				(\tilde p_1,\ldots,\tilde p_M) \eqdef  \Prox_h^{\oKL_\la}( p_1,\ldots,p_M ).
			}
			One has
			\eql{\label{eq-prox-calculus-lifting-transp}
				\Prox_{f}^{\KL_\la}(\pi_1,\ldots,\pi_M) =
				\pa{
					\pi \diag\pa{ \frac{\tilde p_m}{p_m} }
				}_m
			}
\end{proposition}

\newcommand{\proxruleP}[1]{\noindent\textbf{Proof of \eqref{eq-prox-calculus-#1}. }~}
\begin{proof}
		\proxruleP{indic} This is straightforward.  \\
		\proxruleP{shift}  If follows from the relation
		\begin{align*}
			&\oKL_\la( (q_1,\ldots,q_M)|(p_1,\ldots,p_M) ) + \sum_{i=1}^M \dotp{w_i}{q_i} \\ =  
			&\oKL_\la( (q_1,\ldots,q_M)| (p_1 \odot e^{-w_1/\la_1}, \ldots, p_M \odot e^{-w_M/\la_M}) ).
		\end{align*}
		\proxruleP{equal}  
		We denote $(q_m)_m \eqdef \Prox^{\oKL_\la}_{\psi_1}( (p_m)_m )$, so that $q=q_1=\ldots=q_M$ solves
	\eq{
		\umin{q} \sum_m \la_m\oKL(q|p_m) + \tilde h(q).
	}
	The result follow from the relation
	\eq{
		\sum_m \la_m \oKL(q|p_m) = 
		\textstyle \pa{\sum_m \la_m} \oKL\pa{q| p_1^{\tilde\la_1} \odot \ldots \odot  p_M^{\tilde\la_M} 
		}.
	}
		\proxruleP{sum}  Denoting $(q_m)_m = \Prox_{f}^{\oKL_\la}( (p_m)_m )$, the first order optimality condition for $\Prox_{f}^{\oKL_\la}$ reads 
	\eq{
		\foralls m \in \{1,\ldots,M\}, \quad
			\log\pa{\frac{q_m}{p_m}} + u = 0 
	}
	where $u \in \partial h(p_1+\ldots+p_M)$. respectively summing and subtracting these equations lead to
	\eq{
		q_1+\ldots+q_M = \Prox_{h}(p_1+\ldots+p_M)
		\qandq
		\frac{q_1}{p_1} = \ldots = \frac{q_m}{p_m}.
	}
	Solving for $(q_1,\ldots,q_m)$ in these equations leads to the desired solution. \\
	\proxruleP{lifting} The first order condition for $\tilde\pi$ being a solution of~\eqref{eq-defn-proxKL} states the existence of $(z_m)_m \in \partial f(\tilde p_1,\ldots,\tilde p_M)$ where $\tilde p_m = \tilde\pi_m \ones$ such that 
	\eq{
		\la_m \log\pa{ \frac{\tilde\pi_m}{\pi_m} } + z_m \ones^T = 0
		\;\Rightarrow\;
		\tilde\pi_m = \diag(e^{-z_m/\la_m }) \pi_m
		\;\Rightarrow\;
		\tilde p_m = \diag(e^{-z_m/\la_m }) p_m, 
	}	 
	which corresponds to the first order condition for $(\tilde p_m)_m$ being a solution of~\eqref{eq-def-kl-prox-operator} for the function $h$, i.e.
	\eq{
		(\tilde p_m)_m =  \Prox_{h}^{\oKL}( (p_m)_m ).
	} 
	Finally, one obtains
	\eq{
		\tilde\pi_m = \diag(e^{-z_m/\la_m}) \pi_m = \diag\pa{ \frac{\tilde p_m}{p_m} } \pi_m
	}	
	and hence the desired result. 	\\		
		\proxruleP{lifting-transp} It is obtained by transposing formula~\eqref{eq-prox-calculus-lifting}. 
\end{proof}

\bibliographystyle{plain}
\bibliography{bibliography}

\begin{thebibliography}{10}

\bibitem{agueh2002existence}
M.~Agueh.
\newblock {\em Existence of solutions to degenerate parabolic equations via the
  {Monge-Kantorovich} theory.}
\newblock PhD thesis, Georgia Institute of Technology, USA, 2002.

\bibitem{agueh2013one}
M.~Agueh and M.~Bowles.
\newblock One-dimensional numerical algorithms for gradient flows in the
  $p$-{Wasserstein} spaces.
\newblock {\em Acta Applicandae Mathematicae}, 125(1):121--134, 2013.

\bibitem{Carlier_wasserstein_barycenter}
M.~Agueh and G.~Carlier.
\newblock Barycenters in the {W}asserstein space.
\newblock {\em SIAM J. on Mathematical Analysis}, 43(2):904--924, 2011.

\bibitem{ambrosio2006gradient}
L.~Ambrosio, N.~Gigli, and G.~Savar{\'e}.
\newblock {\em Gradient flows: in metric spaces and in the space of probability
  measures}.
\newblock Springer, 2006.

\bibitem{ArtachoNonCvx}
A.~Artacho, Borwein F.~J., and J.~M.
\newblock Global convergence of a non-convex douglas-rachford iteration.
\newblock {\em J. Global Optimization}, 57(3):1--17, 2012.

\bibitem{BauschkeCombettes-Dykstra}
H.~H. Bauschke and P.~L. Combettes.
\newblock A {Dykstra}-like algorithm for two monotone operators.
\newblock {\em Pacific Journal of Optimization}, 4(3):383--391, 2008.

\bibitem{BauschkeCombettes11}
H.~H. Bauschke and P.~L. Combettes.
\newblock {\em Convex Analysis and Monotone Operator Theory in {H}ilbert
  Spaces.}
\newblock Springer-Verlag, New York, 2011.

\bibitem{BauschkeNonCvx}
H.~H. Bauschke, P.~L. Combettes, and D.~R. Luke.
\newblock Phase retrieval, error reduction algorithm, and {Fienup} variants: A
  view from convex optimization.
\newblock {\em J. Opt. Soc. Am. A}, 19(7):1334--1345, 2002.

\bibitem{bauschke-lewis}
H.~H. Bauschke and A.~S. Lewis.
\newblock Dykstra's algorithm with {B}regman projections: a convergence proof.
\newblock {\em Optimization}, 48(4):409--427, 2000.

\bibitem{Benamou2000}
J.-D. Benamou and Y.~Brenier.
\newblock A computational fluid mechanics solution of the {M}onge-{K}antorovich
  mass transfer problem.
\newblock {\em Numerische Mathematik}, 84(3):375--393, 2000.

\bibitem{BregmanProj2015}
J-D. Benamou, G.~Carlier, M.~Cuturi, L.~Nenna, and G.~Peyr\'e.
\newblock Iterative bregman projections for regularized transportation
  problems.
\newblock {\em to appear in SIAM J. Sci. Comp.}, 2015.

\bibitem{JDB-JKO}
J-D. Benamou, G.~Carlier, Q.~M\'erigot, and E.~Oudet.
\newblock Discretization of functionals involving the {Monge-Amp\`ere}
  operator.
\newblock {\em Preprint arXiv:1408.4536}, 2014.

\bibitem{BigotBarycenter}
J.~Bigot and T.~Klein.
\newblock Consistent estimation of a population barycenter in the {W}asserstein
  space.
\newblock {\em Preprint arXiv:1212.2562}, 2012.

\bibitem{blanchet2008convergence}
A.~Blanchet, V.~Calvez, and J.~A Carrillo.
\newblock Convergence of the mass-transport steepest descent scheme for the
  subcritical patlak-keller-segel model.
\newblock {\em SIAM Journal on Numerical Analysis}, 46(2):691--721, 2008.

\bibitem{blanchet2012optimal}
A.~Blanchet and G.~Carlier.
\newblock Optimal transport and {Cournot-Nash} equilibria.
\newblock {\em arXiv preprint arXiv:1206.6571}, 2012.

\bibitem{Bonneel-displacement}
N.~Bonneel, M.~{van de Panne}, S.~Paris, and W.~Heidrich.
\newblock Displacement interpolation using {Lagrangian} mass transport.
\newblock {\em ACM Transactions on Graphics (SIGGRAPH ASIA'11)}, 30(6), 2011.

\bibitem{botsch-2010}
M.~Botsch, L.~Kobbelt, M.~Pauly, P.~Alliez, and B.~Levy.
\newblock {\em Polygon Mesh Processing}.
\newblock Taylor \& Francis, 2010.

\bibitem{bregman1967relaxation}
L.~M. Bregman.
\newblock The relaxation method of finding the common point of convex sets and
  its application to the solution of problems in convex programming.
\newblock {\em USSR computational mathematics and mathematical physics},
  7(3):200--217, 1967.

\bibitem{BregmanCensorReich-Dykstra}
L.M. Bregman, Y.~Censor, and S.~Reich.
\newblock Dykstra's algorithm as the nonlinear extension of {Bregman's}
  optimization method.
\newblock {\em Journal of Convex Analysis}, 6:319--333, 1999.

\bibitem{BrenierEulerAMS}
Y.~Brenier.
\newblock The least action principle and the related concept of generalized
  flows for incompressible perfect fluids.
\newblock {\em J. of the AMS}, 2:225--255, 1990.

\bibitem{BuddMoving}
C.J. Budd, M.J.P. Cullen, and E.J. Walsh.
\newblock {Monge-Amp\`ere} based moving mesh methods for numerical weather
  prediction, with applications to the eady problem.
\newblock {\em Journal of Computational Physics}, 236:247--270, 2013.

\bibitem{burger2010mixed}
M.~Burger, J.~A. Carrillo, and M-T. Wolfram.
\newblock A mixed finite element method for nonlinear diffusion equations.
\newblock {\em Kinetic and Related Models}, 3(1):59--83, 2010.

\bibitem{Burger-JKO}
M.~Burger, M.~Franeka, and C-B. Schonlieb.
\newblock Regularised regression and density estimation based on optimal
  transport.
\newblock {\em Appl. Math. Res. Express}, 2:209--253, 2012.

\bibitem{CarrilloFiniteVolume}
J.~A. Carrillo, A.~Chertock, and Y.~Huang.
\newblock A finite-volume method for nonlinear nonlocal equations with a
  gradient flow structure.
\newblock {\em Communications in Computational Physics}, 17:233--258, 1 2015.

\bibitem{carrillo2009numerical}
J.~A Carrillo and J.~S. Moll.
\newblock Numerical simulation of diffusive and aggregation phenomena in
  nonlinear continuity equations by evolving diffeomorphisms.
\newblock {\em SIAM Journal on Scientific Computing}, 31(6):4305--4329, 2009.

\bibitem{CensorReich-Dykstra}
Y.~Censor and S.~Reich.
\newblock The {Dykstra} algorithm with {Bregman} projections.
\newblock {\em Communications in Applied Analysis}, 2:407--419, 1998.

\bibitem{ChambollePock-div}
A.~Chambolle and T.~Pock.
\newblock On the ergodic convergence rates of a first-order primal-dual
  algorithm.
\newblock {\em preprint}, 2014.

\bibitem{Ciarlet-Book}
P.~G. Ciarlet.
\newblock {\em Introduction to Numerical Linear Algebra and Optimisation}.
\newblock Cambridge University Press, Cambridge, 1989.
\newblock Originally published in French under the title, {\em Introduction \`a
  l'analyse num\'erique matricielle et \`a l'optimisation} in 1982.

\bibitem{crane-2013}
K.~Crane, C.~Weischedel, and M.~Wardetzky.
\newblock Geodesics in heat: A new approach to computing distance based on heat
  flow.
\newblock {\em ACM Trans. Graph.}, 32(5):152:1--152:11, October 2013.

\bibitem{csis}
I.~Csisz{\'a}r.
\newblock {$I$}-divergence geometry of probability distributions and
  minimization problems.
\newblock {\em Ann. Probability}, 3:146--158, 1975.

\bibitem{CuturiSinkhorn}
M.~Cuturi.
\newblock Sinkhorn distances: Lightspeed computation of optimal transport.
\newblock In {\em Advances in Neural Information Processing Systems (NIPS) 26},
  pages 2292--2300, 2013.

\bibitem{CuturiBarycenter}
M.~Cuturi and A.~Doucet.
\newblock Fast computation of {Wasserstein} barycenters.
\newblock In {\em Proceedings of the 31st International Conference on Machine
  Learning (ICML), JMLR W\&CP}, volume~32, 2014.

\bibitem{Davis:2006}
T.~A. Davis.
\newblock {\em Direct Methods for Sparse Linear Systems}.
\newblock SIAM, 2006.

\bibitem{DemingStephanIPFP}
W.~E. Deming and F.~F. Stephan.
\newblock On a least squares adjustment of a sampled frequency table when the
  expected marginal totals are known.
\newblock {\em Annals Mathematical Statistics}, 11(4):427--444, 1940.

\bibitem{deriche-1993}
R.~Deriche.
\newblock Recursively implementing the {G}aussian and its derivatives.
\newblock Technical Report RR-1893, INRIA, 1993.

\bibitem{dyk}
R.~L. Dykstra.
\newblock An iterative procedure for obtaining {$I$}-projections onto the
  intersection of convex sets.
\newblock {\em Ann. Probab.}, 13(3):975--984, 1985.

\bibitem{EcksteinProxPoint}
J.~Eckstein.
\newblock Nonlinear proximal point algorithms using {Bregman} functions, with
  applications to convex programming.
\newblock {\em Mathematics of Operations Research}, 18(1):202--226, 1993.

\bibitem{ErbarHeatManifold}
M.~Erbar.
\newblock The heat equation on manifolds as a gradient flow in the
  {Wasserstein} space.
\newblock {\em Annales de l'Institut Henri Poincar\'e, Probabilités et
  Statistiques}, 46(1):1--23, 2010.

\bibitem{FehrenbachMirebeau}
J.~Fehrenbach and J-M. Mirebeau.
\newblock Sparse non-negative stencils for anisotropic diffusion.
\newblock {\em Journal of Mathematical Imaging and Vision}, 49(1):123--147,
  2014.

\bibitem{FigalliPartial}
A.~Figalli.
\newblock The optimal partial transport problem.
\newblock {\em Arch. Ration. Mech. Anal.}, 195(2):533--560, 2010.

\bibitem{FrischNaturee}
U.~Frisch, S.~Matarrese, R.~Mohayaee, and A.~Sobolevski.
\newblock {Monge-Ampere-Kantorovitch} {(MAK)} reconstruction of the early
  universe.
\newblock {\em Nature}, 417(260), 2002.

\bibitem{JonathanMcCannCapacity}
K.~Jonathan and R.~J. McCann.
\newblock Insights into capacity constrained optimal transport.
\newblock {\em Proc. Natl. Acad. Sci. USA}, 110:10064--10067, 2013.

\bibitem{jordan1998variational}
R.~Jordan, D.~Kinderlehrer, and O.~Otto.
\newblock The variational formulation of the {Fokker}-{Planck} equation.
\newblock {\em SIAM journal on mathematical analysis}, 29(1):1--17, 1998.

\bibitem{Kantorovich42}
L.~Kantorovich.
\newblock On the transfer of masses (in russian).
\newblock {\em Doklady Akademii Nauk}, 37(2):227--229, 1942.

\bibitem{kinderlehrer1999approximation}
D.~Kinderlehrer and N.~J Walkington.
\newblock Approximation of parabolic equations using the {Wasserstein} metric.
\newblock {\em ESAIM: Mathematical Modelling and Numerical Analysis},
  33(04):837--852, 1999.

\bibitem{KiwielProxPoint}
K.C. Kiwiel.
\newblock Proximal minimization methods with generalized {Bregman} functions.
\newblock {\em SIAM J. Control Optim.}, 35(4):1142--1168, 1997.

\bibitem{LeonardShrodinger}
C.~Leonard.
\newblock A survey of the {Schrodinger} problem and some of its connections
  with optimal transport.
\newblock {\em Discrete Contin. Dyn. Syst. A}, 34(4):1533--1574, 2014.

\bibitem{Matthes1D}
D.~Matthes and H.~Osberger.
\newblock Convergence of a variational lagrangian scheme for a nonlinear drift
  diffusion equation.
\newblock {\em Preprint arXiv:1301.0747}, 2014.

\bibitem{maury2010macroscopic}
B.~Maury, A.~Roudneff-Chupin, and F.~Santambrogio.
\newblock A macroscopic crowd motion model of gradient flow type.
\newblock {\em Mathematical Models and Methods in Applied Sciences},
  20(10):1787--1821, 2010.

\bibitem{nesterov1994interior}
Y.~Nesterov, A.~Nemirovskii, and Y.~Ye.
\newblock {\em Interior-point polynomial algorithms in convex programming},
  volume~13.
\newblock SIAM, 1994.

\bibitem{otto2001geometry}
F.~Otto.
\newblock The geometry of dissipative evolution equations: the porous medium
  equation.
\newblock {\em Communications in partial differential equations},
  26(1-2):101--174, 2001.

\bibitem{FPapPeyOud13}
N.~Papadakis, G.~Peyr{\'e}, and E.~Oudet.
\newblock Optimal transport with proximal splitting.
\newblock {\em SIAM Journal on Imaging Sciences}, 7(1):212--238, 2014.

\bibitem{PassMultimarginal}
B.~Pass.
\newblock On the local structure of optimal measures in the multi-marginal
  optimal transportation problem.
\newblock {\em Calc. Var. Partial Differential Equations}, 43(3-4):529--536,
  2012.

\bibitem{RubTomGui00}
Y.~Rubner, C.~Tomasi, and L.J. Guibas.
\newblock The earth mover's distance as a metric for image retrieval.
\newblock {\em International Journal of Computer Vision}, 40(2), 2000.

\bibitem{RuschendorfThomsen}
L.~Ruschendorf and W.~Thomsen.
\newblock Closedness of sum spaces and the generalized {Schrodinger} problem.
\newblock {\em Theory of Probability and its Applications}, 42(3):483--494,
  1998.

\bibitem{Shrodinger31}
E.~Schrodinger.
\newblock Uber die umkehrung der naturgesetze.
\newblock {\em Sitzungsberichte Preuss. Akad. Wiss. Berlin. Phys. Math.},
  144:144--153, 1931.

\bibitem{sethian-book}
J.A. Sethian.
\newblock {\em Level Sets Methods and Fast Marching Methods}.
\newblock Cambridge University Press, 2nd edition, 1999.

\bibitem{Sinkhorn64}
R.~Sinkhorn.
\newblock A relationship between arbitrary positive matrices and doubly
  stochastic matrices.
\newblock {\em Ann. Math. Statist.}, 35:876--879, 1964.

\bibitem{Sinkhorn67}
R.~Sinkhorn.
\newblock Diagonal equivalence to matrices with prescribed row and column sums.
\newblock {\em Amer. Math. Monthly}, 74:402--405, 1967.

\bibitem{SinkhornKnopp67}
R.~Sinkhorn and P~. Knopp.
\newblock Concerning nonnegative matrices and doubly stochastic matrices.
\newblock {\em Pacific J. Math.}, 21:343--348, 1967.

\bibitem{SolodovSvaiterBregman}
M.~V. Solodov and B.~F. Svaiter.
\newblock An inexact hybrid generalized proximal point algorithm and some new
  results on the theory of {Bregman} functions.
\newblock {\em Mathematics of Operations Research}, 25(2):214--230, 2000.

\bibitem{ConvolutionalOT}
J.~Solomon, F.~{de Goes}, G.~Peyr\'e, M.~Cuturi, A.~Butscher, A.~Nguyen, T.~Du,
  and L.~Guibas.
\newblock Convolutional {Wasserstein} distances: Efficient optimal
  transportation on geometric domains.
\newblock {\em Preprint}, 2015.

\bibitem{varadhan-1967}
S.~R.~S. Varadhan.
\newblock On the behavior of the fundamental solution of the heat equation with
  variable coefficients.
\newblock {\em Communications on Pure and Applied Mathematics}, 20(2):431--455,
  1967.

\bibitem{Villani03}
C.~Villani.
\newblock {\em Topics in Optimal Transportation}.
\newblock Graduate Studies in Mathematics Series. American Mathematical
  Society, 2003.

\bibitem{WangBanerjee-ADMM}
H.~Wang and A.~Banerjee.
\newblock Bregman alternating direction method of multipliers.
\newblock {\em preprint arXiv:1306.3203}, 2014.

\bibitem{WeickertBook}
J.~Weickert.
\newblock {\em Anisotropic Diffusion in Image Processing}.
\newblock Teubner, Stuttgart, 1998.

\bibitem{Westdickenberg2010}
M.~Westdickenberg and J.~Wilkening.
\newblock Variational particle schemes for the porous medium equation and for
  the system of isentropic {Euler} equations.
\newblock {\em ESAIM: Mathematical Modelling and Numerical Analysis},
  44(1):133--166, 2010.

\bibitem{2014-xia-siims}
G-S. Xia, S.~Ferradans, G.~Peyr{\'e}, and J-F. Aujol.
\newblock Synthesizing and mixing stationary {Gaussian} texture models.
\newblock {\em SIAM Journal on Imaging Sciences}, 7(1):476--508, 2014.

\end{thebibliography}

\end{document}